\documentclass[12pt,leqno,twoside]{amsart}
\usepackage{amssymb}
\usepackage{latexsym}
\usepackage{graphicx}
\usepackage{epsfig}
\usepackage{srcltx}
\usepackage[colorlinks=true, pdfstartview=FitV, linkcolor=blue, citecolor=blue, urlcolor=blue]{hyperref}
\usepackage{subcaption}

\usepackage[margin=1.04in]{geometry}

\newtheorem{theorem}{Theorem}[section]
\newtheorem{corollary}[theorem]{Corollary}
\newtheorem{lemma}[theorem]{Lemma}
\newtheorem{proposition}[theorem]{Proposition}

\theoremstyle{remark}
\newtheorem{remark}[theorem]{\sc Remark}
\theoremstyle{remark}

\theoremstyle{definition}
\newtheorem{definition}[theorem]{Definition}
\theoremstyle{remark}
\newtheorem{example}[theorem]{\sc Example}

\theoremstyle{remark}

\theoremstyle{remark}

\numberwithin{equation}{section}  

\renewcommand{\Box}{\square}    

\newcommand{\cal}{\mathcal}

\newcommand{\cFb}{\mathcal P}
\newcommand{\cGb}{{\mathcal G}^{\bullet}}

\def\be{\begin{equation}}
\def\ee{\end{equation}}
\def\bt{\begin{theorem}}
\def\et{\end{theorem}}
\newcommand{\lra}{\longrightarrow}
\def\bc{\begin{corollary}}
\def\ec{\end{corollary}}

\newcommand{\rank}{\mathrm{rank\hspace{2pt}}}
\renewcommand{\r}{{\mathrm{rHd}}}

\renewcommand{\int}{{\mathrm{int}}}
\newcommand{\mult}{{\mathrm{mult}}}

\renewcommand{\int}{{\mathrm{int}}}
\newcommand{\Aut}{\mathop{{\rm{Aut}}}\nolimits}

\newcommand{\Sing}{{\mathrm{Sing\hspace{0.5pt}}}}

\newcommand{\id}{{\mathrm{id}}}

\newcommand{\im}{{\mathrm{Im\hspace{1pt}}}}
\renewcommand{\ker}{\mathop{{\mathrm{ker}}}\nolimits}
\newcommand{\coker}{\mathop{{\mathrm{coker}}}\nolimits}
\newcommand{\lk}{{\mathbb C \rm{lk}}}

\newcommand{\e}{\varepsilon}
\newcommand{\m}{\setminus}

\newcommand{\fin}{\hspace*{\fill}$\Box$\vspace*{2mm}}

\newcommand{\tF}{F^{\pitchfork}}
\newcommand{\tE}{E^{\pitchfork}}
\newcommand{\tmu}{\mu^{\pitchfork}}

\newcommand{\cC}{{\cal C}}

\newcommand{\cH}{{\cal H}}

\newcommand{\cF}{{\cal F}}

\newcommand{\cS}{{\cal S}}

\newcommand{\cW}{{\cal W}}

\newcommand{\cZ}{{\cal Z}}

\newcommand{\bC}{{\mathbb C}}

\newcommand{\bZ}{{\mathbb Z}}

\newcommand{\bQ}{{\mathbb Q}}

\newcommand{\bH}{{\mathbb H}}

\begin{document}

\title[The vanishing cohomology of singular hypersurfaces]
{The vanishing cohomology of non-isolated hypersurface singularities}

\author[L. Maxim ]{Lauren\c{t}iu Maxim}
\address{L. Maxim : Department of Mathematics, University of Wisconsin-Madison, 480 Lincoln Drive, Madison WI 53706-1388, USA}
\email {maxim@math.wisc.edu}
\thanks{L. Maxim is partially supported by the Simons Foundation Collaboration Grant \#567077.}

\author[L. P\u{a}unescu ]{Lauren\c{t}iu P\u{a}unescu}
\address{L. P\u{a}unescu: Department of Mathematics, University of Sydney, Sydney, NSW, 2006, Australia}
\email {laurentiu.paunescu@sydney.edu.au}

\author[M. Tib\u{a}r]{Mihai Tib\u{a}r}
\address{M. Tib\u{a}r : Universit\' e de  Lille, CNRS, UMR 8524 -- Laboratoire Paul Painlev\'e, F-59000 Lille, France}  
\email {mihai-marius.tibar@univ-lille.fr}
\thanks{M. Tib\u{a}r acknowledges the support of the Labex CEMPI (ANR-11-LABX-0007). }



\date{\today}



\begin{abstract}
We employ the perverse vanishing cycles to 
show that each reduced cohomology group of the Milnor fiber, except the top two, can be computed from the restriction of the vanishing cycle complex to only singular strata with a certain lower bound in dimension.
Guided by geometric results, we alternately use the nearby and vanishing cycle functors to derive information about the Milnor fiber cohomology via iterated slicing by generic hyperplanes. These lead to the description of the reduced cohomology groups, except the top two, in terms of the vanishing cohomology of the nearby section. We use it to compute explicitly the lowest (possibly nontrivial) vanishing cohomology group of the Milnor fiber. 
\end{abstract}
 
\maketitle

\setcounter{section}{0}

\section{Introduction}

In his search for exotic spheres, Milnor \cite{Mi} initiated the study of the topology of complex hypersurface singularity germs. For a germ of an analytic map $f:(\bC^{n+1},0) \to (\bC,0)$ having a singularity at the origin, Milnor introduced what is now called the Milnor fibration $f^{-1}(D_\delta^*)\cap B_\epsilon \to D_\delta^*$ (in a small enough ball $B_\epsilon \subset \bC^{n+1}$ and over a small enough punctured disc $D_\delta^* \subset \bC$), and the Milnor fiber $F:=f^{-1}(t) \cap B_\epsilon$  of $f$ at $0$.  
Around the same time, Grothendieck  \cite{Gr1} proved Milnor's conjecture that the eigenvalues of the monodromy acting on $H^*(F;\bZ)$ are roots of unity, and Deligne \cite{Gr2} defined the nearby and vanishing cycle functors, $\psi_f$ and $\varphi_f$, globalizing Milnor's construction. 
A few years later, L\^e \cite{Le2} extended the geometric setting of the Milnor tube fibration to the case of functions defined on complex analytic germs, using the existence of Thom-Whitney stratifications on the zero set $f^{-1}(0)$, which was proved by Hironaka \cite{Hi}. 

Since their introduction more than half a century ago, the Milnor fiber  and vanishing cycles have found a wide range of applications, in fields like algebraic geometry, algebraic and geometric topology, symplectic geometry, singularity theory, enumerative geometry, computational topology and algebraic statistics, etc.
However, despite the enormous interest and vast applications, the very basic question of describing the topology of  the Milnor fiber (e.g., Betti numbers) for arbitrary hypersurface singularity germs remains largely open.

The Milnor fiber of an isolated hypersurface singularity germ was completely described by Milnor in \cite{Mi}. The study of hypersurfaces with $1$-dimensional singularities was initiated by Yomdin \cite{Io},  Siersma \cite{Si1, Si-icis, Si3, Si-Cam}, and continued and refined by  Vannier \cite{Va}, Pellikaan \cite{Pe0,Pe}, Schrauwen \cite{Sch}, de Jong \cite{deJ}, Zaharia \cite{Za0}, Tib\u{a}r \cite{Ti-iomdin, ST-deform}, etc. Milnor fibers for higher dimensional singularities were studied by L\^e  \cite{Le0}, Zaharia \cite{Za}, Shubladze \cite{Sh1,Sh2}, Massey \cite{Ma91, Ma94, Ma95, Ma0, Ma4}, Nemethi \cite{Ne, Ne2}, Tib\u{a}r \cite{Ti-nonisol, Ti-book}, Dimca-Saito \cite{DS}, Libgober \cite{Li}, Maxim \cite{Max1}, Fernandez de Bobadilla \cite{Fe, Fe2}, etc. 

In \cite{Le0}, L\^e developed a recurrent method for computing the Betti numbers of the Milnor fiber, based on slicing the singularity by generic hyperplanes, and in \cite{Va}, Vannier described a general handlebody model for the Milnor fibre of function germs with a $1$-dimensional singular locus. These results led Massey to a certain estimation of the Betti numbers of the  Milnor fiber,  based on the polar multiplicities encoded into what he called the ``L\^e numbers'' (e.g., see \cite{Ma95}, and also \cite{Ma0} for further developments). 

In \cite{DS}, Dimca and Saito investigated local consequences of the perversity of vanishing cycles, and computed the Milnor fiber cohomology from the restriction of the vanishing cycle complex to the real link of the singularity. In particular, they show that the reduced cohomology groups $\widetilde{H}^i(F;\bQ)$ of the Milnor fiber are completely determined for $i<n-1$ (and for $i=n-1$ only partially) by the restriction of the vanishing cycle complex to the complement of the singularity. More precise computations are made in the case when $f^{-1}(0)$ is a divisor with normal crossings in a punctured neighborhood of the singular point, and results are also given on the size of the Jordan blocks of the monodromy in this particular case. If $f$ is a homogeneous polynomial, a more refined dependence of the vanishing cohomology on the singular strata was obtained by Maxim in \cite[Proposition 5.1]{Max1}, and also Libgober \cite[Theorem 3.1]{Li} in the case when $f$ defines a central hyperplane arrangement.

\medskip

In this paper, we study the Milnor fiber cohomology for hypersurface singularity germs with a singular locus of arbitrary positive dimension. Before discussing our main results, let us introduce some notation and assumptions. Unless otherwise specified, all cohomology groups in this paper are with $\bZ$-coefficients.

\medskip

For $n \geq 1$, we consider a nonconstant holomorphic function germ $f:(X,0) \to (\bC,0)$ defined on a pure $(n+1)$-dimensional complex singularity germ $(X,0)$ contained in some ambient $(\bC^N,0)$. We assume moreover that
$$\r(X,\bZ)=n+1,$$ with $\r(X,\bZ)$ denoting the {\it rectified homological depth} of $X$ with respect to the ring $\bZ$, as defined by Grothendieck \cite{SGA2}, see also \cite{HL}, \cite{Sc}. For example, a local complete intersection  $X$ of pure complex dimension $n+1$ satisfies this property. 
Our assumption on $\r(X,\bZ)$ is equivalent to the following two conditions: $(i)$ the shifted constant sheaf $\underline{\bZ}_X[n+1]$ is a {\it $\bZ$-perverse sheaf} on $X$, and $(ii)$ the costalks of $\underline{\bZ}_X[n+1]$ in the lowest possible degree are free abelian; see Corollary \ref{cfield} for a precise formulation of this equivalence. If one is only interested in the $\bQ$-cohomology of the Milnor fiber (e.g., Betti numbers), it suffices to assume that $\r(X,\bQ)=n+1$, which is equivalent to the fact that $\underline{\bQ}_X[n+1]$ is a {\it $\bQ$-perverse sheaf} on $X$.

In the formulation of our results, the notation $X$ will tacitly mean a small enough representative of the germ at $0$ of the space $X$, that is, the intersection of $X$ with an arbitrarily small ball at the origin. The same convention applies to all the subgerms of $(X,0)$.
We denote by $\Sigma$ the germ of the stratified singular locus of $f$ (with respect to some fixed Whitney stratification of $X$), and we assume that the complex dimension $s$ at the origin of $\Sigma$ satisfies $0<s<n$. We denote by $F$ the Milnor fiber of $f$ at the origin, and let $\cFb:=\varphi_f\underline{\bZ}_X[n]$ be the complex of perverse vanishing cycles. It is then well known that the reduced Milnor fiber cohomology (i.e., the {\it vanishing cohomology}) of $F$ is computed by $\cFb$, and the only possibly non-trivial vanishing cohomology groups $\widetilde H^k(F)$ are concentrated in degrees $n-s \leq k \leq n$. 
Furthermore, the assumption $\r(X,\bZ)=n+1$ yields that the $\bZ$-perverse sheaf $\cFb$ also satisfies the property of freeness of costalks in lowest possible degree. This property is used in our paper to formulate cohomological substitutes for various homotopy-theoretic statements for the Milnor fiber (which don't necessarily hold if $X$ is singular).

 In the above notations, our first result (Theorem \ref{t14a}) uses only the perversity of the vanishing cycles $\cFb$ to show that, if $s \geq 2$ and $k < n-1$, the group $\widetilde{H}^{k}(F)$ can be computed from the restriction of the vanishing cycle complex to the singular strata of dimension $\geq n-k-1$ (see also Corollary \ref{c2.2} and Proposition \ref{t:izos}). This is a strengthening of the above mentioned result of Dimca-Saito (but from a different perspective, in the sense that our result involves the intersection of the whole singular locus of $f$ with a sufficiently small open ball, while Dimca-Saito formulate their result in terms of the real link; see Remark \ref{dif}(b) for more details), and it reduces the calculation of the vanishing cohomology to a hypercohomology spectral sequence. 
This computation, while tedious in general, can be made more explicit in the case of the least nontrivial vanishing cohomology group $\widetilde H^{n-s}(F)$. In fact, if $s\geq 1$, Theorem  \ref{c14a}  gives a monomorphism 
\be\label{upbi}
\widetilde{H}^{n-s}(F) \hookrightarrow \bigoplus_i \widetilde{H}^{n-s}(F^\pitchfork_{s,i})^{A_i},
\ee
where the summation is over the collection $\{\Sigma_{s,i}\}_i$  of $s$-dimensional singular strata of the germ of $\Sigma$ at the origin, $F^\pitchfork_{s,i}$ is the transversal Milnor fiber to the $s$-dimensional stratum $\Sigma_{s,i}$, and $A_i$ denotes the action of $\pi_1(\Sigma_{s,i})$ on $\widetilde{H}^{n-s}(F^\pitchfork_{s,i})$, with invariant subspaces appearing on the right hand side of \eqref{upbi}.
This represents an extension to arbitrary singularities of Siersma's  results  \cite{Si3} for $1$-dimensional singularities in terms of the {\it vertical monodromy},  as well as of the more recent development in \cite{ST-deform}, and answers Siersma's conjecture made in his overview \cite{Si-Cam}. It also provides 
 an upper bound for the Betti number $b_{n-s}(F)$. Note that the assumption $\r(X,\bZ)=n+1$ yields in addition that $\widetilde{H}^{n-s}(F)$ is free, since the right-hand side of \eqref{upbi} becomes in this case free.
 
If $s \geq 2$, Theorem \ref{c14a}(b) (see also \eqref{stl3s} and Remark \ref{r65}) expresses the ``defect'' of \eqref{upbi} from being an isomorphism in terms of costalks of the perverse vanishing cycle complex $\cFb$ at points in the $(s-1)$-dimensional strata. In particular, this yields a lower bound for $b_{n-s}(F)$, which to our knowledge is completely new. We also observe that, if $s\geq 2$ and $\Sigma$ does not contain any $(s-1)$-dimensional strata, the monomorphism \eqref{upbi} becomes an isomorphism. 
 
 If $s\geq 1$, monomorphisms similar to \eqref{upbi} are obtained in Theorem \ref{t14a} for all vanishing cohomology groups $\widetilde{H}^{k}(F)$, $n-s \leq k \leq n-1$, though not as explicit as in the case $k=n-s$.
The monomorphism \eqref{upbi} leads to divisibility results for the characteristic polynomials of Milnor monodromies (Corollary \ref{c:2}), as well as to upper bounds for the dimension of the monodromy eigenspaces and maximal sizes of Jordan blocks (Corollary \ref{c:3}). Furthermore, similar statements can be formulated for the eigenspaces of vanishing cohomology in terms of the generalized eigensheaves of the vanishing cycles (Theorem \ref{t15a} and Corollary \ref{c15a}).

\medskip

We next take a more geometric viewpoint  for computing the Milnor fiber cohomology via slicing.
In \S\ref{sl} we present a sheaf theoretic interpretation of results of L\^e \cite{Le0} and Tib\u{a}r \cite{Ti-nonisol} on the computation of the number of vanishing cycles via iterated slicing by generic hyperplanes.
In Proposition \ref{t:izos}, we then give (for $s \geq 2$) an interpretation of the reduced Milnor fiber cohomology in terms of the vanishing cohomology of the nearby section of $f$ (introduced in  \cite{Ti-nonisol}). 
Upper bounds for the Betti numbers of the Milnor fiber are derived via slicing in terms of polar multiplicities (Corollary \ref{t:bettibounds}). 

This slicing technique enables us to develop in  \S\ref{s:Hn-s}  a geometric method for the computation of the lowest  (possibly nontrivial)  vanishing cohomology group $H^{n-s}(F)$, for $n>s\ge 2$,  by using its isomorphism with the vanishing cohomology of the nearby section. It turns out that, by repeated slicing, the computation reduces to the case of a $2$-dimensional singular locus.  We then use ideas developed in  \cite{ST-vanishing} and \cite{ST-deform} for computing the  homology groups in case of 1-dimensional singularities via admissible deformations of $f$ to obtain in Corollary \ref{c:no1diminside} and Theorem \ref{t:euler} a description of $H^{n-s}(F)$ in terms of the invariant submodule of a monodromy representation \eqref{eq:hatA} on the transversal Milnor fibre of the $s$-dimensional strata,  which is in fact a genuine vertical monodromy representation. The similarity with the terms and results of Theorem \ref{c14a} is striking.
 But surprisingly, this monodromy representation is totally different from \eqref{upbi} used in Theorem \ref{c14a}, and in general cannot be deduced from it by a Lefschetz slicing argument.
 
  The results obtained from the perverse sheaf side and from the geometric side feed and complete each other in the pursuit of computability.
We detail a bunch of examples in \S \ref{examples}, showing how the general theory works and emphasizing the efficacy of Theorem \ref{t:euler} in computing $H^{n-s}(F)$.


\section{Preliminaries}\label{prel}
In this section, we review several concepts which play an important role in the remainder of the paper. In \S\ref{pv}, we recall some background on constructible complexes and perverse sheaves. In \S\ref{rect}, we recall the notion of rectified homological depth and indicate its relation to perverse sheaves. In \S\ref{nv}, we introduce the nearby and vanishing cycle functors associated to a holomorphic map, and describe their relation to the cohomology of the Milnor fiber. 


\subsection{Perverse sheaves}\label{pv}
In this section, we recall the definition of perverse sheaves (e.g., see \cite{Di, Max} for a quick introduction).

Let $A$ be a noetherian commutative ring of finite global dimension. Let $X$ be a complex analytic variety, and denote by $D^b(X)$ the derived category of bounded complexes of sheaves of $A$-modules. 

Recall that  a sheaf $\cF$ of $A$-modules is said to be {\it constructible} if there is a Whitney stratification $\cW$ of $X$ so that the restriction $\cF\vert_S$ of $\cF$ to every stratum $S\in \cW$ is an $A$-local system with finitely generated stalks. A bounded complex $\cF^{\centerdot} \in D^b(X)$ is said to be constructible if all its cohomology sheaves $\cH^j(\cF^{\centerdot})$ are constructible.
Denote by $D^b_c(X)$ the full triangulated subcategory of $D^b(X)$ consisting of constructible complexes (i.e., complexes which are constructible with respect to some Whitney stratification).

All dimensions below are taken to be complex dimensions. 
\begin{definition}\label{sper} \ 
\noindent(a) \ The {\it perverse $t$-structure} on $D_c^b(X)$ consists of the subcategories ${^pD}^{\leq 0}(X)$ and ${^pD}^{\geq 0}(X)$ of $D_c^b(X)$ defined as:
$${^pD}^{\leq 0}(X)=\{ \cF^{\centerdot} \in D^b_c(X) \mid \dim {\rm supp}^{-j}(\cF^{\centerdot}) \leq j, \forall j \in \bZ \},$$
$${^pD}^{\geq 0}(X)=\{ \cF^{\centerdot} \in D^b_c(X) \mid \dim {\rm cosupp}^{j}(\cF^{\centerdot}) \leq j, \forall j \in \bZ \},$$
where, for $k_x:\{x\} \hookrightarrow X$ denoting the point inclusion, we define the {\it $j$-th support} and, resp., the {\it $j$-th cosupport} of $\cF^{\centerdot} \in D^b_c(X)$ by:
$${\rm supp}^{j}(\cF^{\centerdot}) = \overline{\{x \in X \mid H^j(k_x^*\cF^{\centerdot}) \neq 0\}},$$
$${\rm cosupp}^{j}(\cF^{\centerdot}) = \overline{\{x \in X \mid H^j(k_x^!\cF^{\centerdot}) \neq 0\}}.$$
Here, $k_x^*\cF^{\centerdot}$ and $k_x^!\cF^{\centerdot}$ are called the {\it stalk} and, resp., {\it costalk} of $\cF^{\centerdot}$ at $x$. \\
\noindent(b) A  complex $\cF^{\centerdot} \in D^b_c(X)$ is called a {\it perverse sheaf} on $X$ if $\cF^{\centerdot} \in {^pD}^{\leq 0}(X) \cap {^pD}^{\geq 0}(X).$ \\
\noindent(c) We say that $\cF^{\centerdot} \in D^b_c(X)$ is {\it strongly perverse} if $\cF^{\centerdot} \in {^pD}^{\leq 0}(X) $ and  $\mathcal{D}_X\cF^{\centerdot} \in {^pD}^{\leq 0}(X)$, with $\mathcal{D}_X$ denoting the dualizing functor.
\end{definition}

\begin{example}\label{e22} Assume $X$ is of pure complex dimension with $c:X \to {\rm pt}$ the constant map to a point space, and let $\underline{A}_X=c^*A$ be the constant $A$-sheaf on $X$. Then:
\begin{itemize}
\item[(a)] $\underline{A}_X[\dim X] \in {^pD}^{\leq 0}(X)$.
\item[(b)] If $X$ is a local complete intersection then $\underline{A}_X[\dim X]$ is a perverse sheaf on $X$ (e.g., see \cite[Theorem 5.1.20]{Di}).
\end{itemize}
\end{example}

It is important to note that the categories ${^pD}^{\leq 0}(X)$ and ${^pD}^{\geq 0}(X)$ can also be described in terms of a fixed Whitney stratification of $X$. Indeed, the perverse $t$-structure can be characterized as follows (e.g., see \cite[Theorem 8.3.1]{Max}):
\begin{theorem}
If $\cF^{\centerdot} \in D^b_c(X)$ is constructible with respect to a Whitney stratification $\cW$ of $X$, then:
\begin{itemize}
\item[(i)] stalk vanishing:
\begin{center}
$\cF^{\centerdot} \in {^pD}^{\leq 0}(X)$ $\iff$ $\forall S \in \cW$,  $\forall x \in S$: $H^j(k_x^*\cF^{\centerdot})=0$ for all $j>-\dim S$.
\end{center}
\item[(ii)] costalk vanishing:
\begin{center}
$\cF^{\centerdot} \in {^pD}^{\geq 0}(X)$ $\iff$ $\forall S \in \cW$,  $\forall x \in S$: $H^j(k_x^!\cF^{\centerdot})=0$ for all $j<\dim S$.
\end{center}
\end{itemize}
\fin
\end{theorem}

\begin{remark}\label{field} \
\noindent(a) If $S \in \cW$ is a stratum of $X$ with inclusion map $k_S:S\hookrightarrow X$, the above costalk vanishing condition for $S$ is equivalent to: 
\be\label{coss} \cH^j(k_S^!\cF^{\centerdot})=0, \ \ {\rm for \ all} \  j<-\dim S.\ee
\noindent(b) If $A$ is a field, the notions of perverse sheaf and strongly perverse sheaf coincide. Indeed, in this case, the universal coefficient theorem yields that $\cF^{\centerdot} \in {^pD}^{\geq 0}(X)$ if and only if $\mathcal{D}_X\cF^{\centerdot} \in {^pD}^{\leq 0}(X)$.
\end{remark}

We conclude this section with the following. 
\begin{proposition}\label{pid}
Assume that the ring $A$ is a principal ideal domain (e.g., $A=\bZ$). If $\cF^{\centerdot} \in D^b_c(X)$ is constructible with respect to a Whitney stratification $\cW$ of $X$, then $\mathcal{D}_X\cF^{\centerdot} \in {^pD}^{\leq 0}(X)$ if and only if the following two conditions are satisfied:
\begin{itemize}
\item[(i)] $\cF^{\centerdot} \in {^pD}^{\geq 0}(X)$;
\item[(ii)] for any stratum $S \in \cW$ and any $x \in S$, the costalk cohomology $H^{\dim S}(k_x^!\cF^{\centerdot})$ is free.
\end{itemize}
In particular, $\cF^{\centerdot}$ is strongly perverse (with respect to $\cW$) if and only if $\cF^{\centerdot}$ is perverse and property (ii) above holds (i.e., costalks of $\cF^{\centerdot}$  in the lowest possible degree are free on each stratum).
\end{proposition}
\begin{proof} Let $S \in \cW$ and $x \in S$, with inclusion $k_x:\{x\} \hookrightarrow X$.
Properties of the dualizing functor and the universal coefficient theorem yield:
$$H^j(k_x^*\mathcal{D}_X\cF^{\centerdot})\cong H^j(\mathcal{D}_Xk_x^!\cF^{\centerdot})
\cong {\rm Hom}(H^{-j}(k_x^!\cF^{\centerdot}), A) \oplus {\rm Ext}(H^{-j+1}(k_x^!\cF^{\centerdot}), A).$$
The desired equivalence can now be checked easily.
\end{proof}


\subsection{Rectified homological depth}\label{rect}
In this section we recall the notion of rectified homological depth, and indicate its relation to (strongly) perverse sheaves. 

Let $X$ be a complex analytic space of complex pure dimension. Following \cite[Definition 6.0.4]{Sc}, we make the following.
\begin{definition}\label{defrhd} The {\it rectified homological depth} $\r(X,A)$ of $X$ with respect to the commutative base ring $A$ is $\geq d$ (for some $d \in \bZ$) if 
\be\label{rhdd}
\mathcal{D}_X(\underline{A}_X [d]) \in {^pD}^{\leq 0}(X).
\ee
\end{definition}
As pointed out in \cite{Sc}, the above definition agrees with the notion of rectified homological depth introduced by Hamm and L\^e \cite{HL} in more geometric terms.
\begin{example}\label{ex27}
\begin{itemize}
\item[(a)] One always has $\r(X,A) \leq \dim(X)$, and $\r(X,A) = \dim(X)$ if $X$ is smooth and nonempty.
\item[(b)] If $X$ is a pure-dimensional local complete intersection, then $\r(X,A)=\dim X$ (cf. \cite[Example 6.0.11]{Sc}).
\end{itemize}
\end{example}
In view of Example \ref{e22}(a) and Definition \ref{sper}(c), we have the following equivalence (see also \cite[(6.14)]{Sc}):
\begin{proposition}\label{p28}
For any nonempty pure-dimensional complex analytic space $X$, we have:
\begin{center}
$rHd(X,A)=\dim X$ $\iff$ $\underline{A}_X [\dim X]$ is strongly perverse.
\end{center}
\fin
\end{proposition}
As a consequence, Remark \ref{field}(b) and Proposition \ref{pid} yield the following.
\begin{corollary}\label{cfield}
Let $X$ be a nonempty pure-dimensional complex analytic space with a Whitney stratification $\cW$.
\begin{itemize}
\item[(a)] If $A$ is a field, then:
\begin{center}
$rHd(X,A)=\dim X$ $\iff$ $\underline{A}_X [\dim X]$ is perverse.
\end{center}
\item[(b)] If $A$ is a principal ideal domain, then $rHd(X,A)=\dim X$ if and only if the following two conditions are satisfied:
\begin{itemize}
\item[(i)] $\underline{A}_X [\dim X]$ is perverse.
\item[(ii)] for any stratum $S \in \cW$ and any $x \in S$ with $k_x:\{x\} \hookrightarrow X$, the costalk cohomology $H^{\dim S}(k_x^!\underline{A}_X[\dim X])$ is free.
\end{itemize}
\end{itemize}
\fin
\end{corollary}

 \subsection{Nearby and vanishing cycle functors} \label{nv} \ 
Let $n \geq 1$, and consider a nonconstant holomorphic function germ $f:(X,0) \to (\bC,0)$ defined on a pure $(n+1)$-dimensional complex singularity germ $(X,0)$ contained in some ambient $(\bC^N,0)$. Assume as in \cite{Ti-nonisol} that \be \r(X,\bZ)=n+1,\ee with $\r(X,\bZ)$ denoting the {\it rectified homological depth} of $X$ with respect to the ring $\bZ$ 
(cf. Definition \ref{defrhd}). As seen in Example \ref{ex27}(b), a local complete intersection  $X$ of pure complex dimension $n+1$ satisfies this property. By Proposition \ref{p28}, our assumption on $\r(X,\bZ)$ is equivalent to the condition that the shifted constant sheaf $\underline{\bZ}_X[n+1]$ is a {\it strongly perverse sheaf} on $X$, which by Corollary \ref{cfield} means that: $(i)$ the shifted constant sheaf $\underline{\bZ}_X[n+1]$ is a perverse sheaf on $X$ in the usual sense, and $(ii)$  on each stratum, the costalks of $\underline{\bZ}_X[n+1]$ in the lowest possible degree are free abelian. (Moreover, the weaker condition $\r(X,\bQ)=n+1$ is equivalent to the fact that $\underline{\bQ}_X[n+1]$ is a $\bQ$-perverse sheaf on $X$, which is the hypothesis considered in \cite{DS}.)

Since $Y:=f^{-1}(0)$ is a principal divisor on $X$, it follows from \cite[Example 6.0.11]{Sc} that $\r(Y,\bZ)=n$, and hence 
$\underline{\bZ}_Y[n]$ is a strongly perverse sheaf on $Y$ in the above sense.

We denote by $\psi_f, \varphi_f$ the nearby and vanishing cycle functors associated to $f$ (e.g., see \cite{Di, Max} for a quick introduction). Recall that if $u: Y=f^{-1}(0) \hookrightarrow X$ is the inclusion map, there is a distinguished triangle of functors
\be\label{uno} u^* \to \psi_f \to \varphi_f \overset{[1]}{\to} \ee
Moreover, the nearby and vanishing cycle complexes $\psi_f\underline{\bZ}_X$ and $\varphi_f\underline{\bZ}_X$ encode the Milnor fiber cohomology at points along $Y$, in the sense that
\be\label{i62} H^k(F_x) \cong \cH^k(\psi_f\underline{\bZ}_X)_x  \ \ \ \text{and} \ \ \ \widetilde{H}^k(F_x) \cong \cH^k(\varphi_f\underline{\bZ}_X)_x,
\ee
with $F_x$ denoting the Milnor fiber of $f$ at $x \in Y$, and where $\cH^*(-)_x$ computes the stalk cohomology at $x$. In particular, the vanishing cycle complex $\varphi_f\underline{\bZ}_X$ is supported on the stratified singular locus ${\rm Sing}_{\cW}(f)$ of $f$ with respect to a fixed Whitney stratification $\cW$ of $X$, which is contained in $Y$ as a set germ at the origin.

It is also known that the shifted functors $^p\psi_f:=\psi_f[-1]$ and $^p\varphi_f:=\varphi_f[-1]$
preserve strongly perverse sheaves (e.g., see \cite[Theorem 6.0.2, Remark 6.0.6]{Sc}), so under our assumptions the shifted nearby and vanishing cycle complexes $\psi_f\underline{\bZ}_X[n]$ and, resp., $\varphi_f\underline{\bZ}_X[n]$ are strongly perverse sheaves on $Y$. In particular, these perverse sheaves have free costalks in the lowest possible degree. We refer to $^p\psi_f$, $^p\varphi_f$ as the perverse nearby and vanishing cycle functors.

If we work with $\bC$-coefficients, let us denote by $h$ the monodromy of $\varphi_f\underline{\bC}_X$, with its Jordan decomposition $h=h_uh_s$, where $h_s$ is semi-simple (and locally of finite order) and $h_u$ is unipotent.
For any $\lambda \in \bC$, we set $$\psi_{f,\lambda}\underline{\bC}_X:=\ker (h_s - \lambda) \subset \psi_f\underline{\bC}_X$$ and similarly for $\varphi_{f,\lambda}\underline{\bC}_X$, in the category of shifted perverse sheaves. It follows from the definition of vanishing cycles that 
$\psi_{f,\lambda}\underline{\bC}_X=\varphi_{f,\lambda}\underline{\bC}_X$ for $\lambda \neq 1$. Moreover, we have decompositions
$$\psi_{f}\underline{\bC}_X =\bigoplus_{\lambda} \psi_{f,\lambda}\underline{\bC}_X, \ \ 
\varphi_{f}\underline{\bC}_X =\bigoplus_{\lambda} \varphi_{f,\lambda}\underline{\bC}_X,$$
and for $x \in Y$ we have isomorphisms
\be \label{eq:defvanish}
H^k(F_x;\bC)_\lambda \cong \cH^k(\psi_{f,\lambda}\underline{\bC}_X)_x \ \ \ \text{and} \ \ \ 
 \widetilde{H}^k(F_x;\bC)_{\lambda} \cong \cH^k(\varphi_{f,\lambda}\underline{\bC}_X)_x,
 \ee
with ${H}^k(F_x;\bC)_{\lambda}$ the corresponding $\lambda$-eigenspace of the action of the semi-simple part of the Milnor monodromy on ${H}^k(F_x;\bC)$.


\section{Milnor fiber cohomology via perverse vanishing cycles}\label{perv}
In this section, we investigate the cohomology of the Milnor fiber 
as a consequence of the fact that, up to a shift, the nearby and vanishing cycle functors preserve perverse sheaves. In \S\ref{bound} we  show that each reduced cohomology group of the Milnor fiber, except the top two, can be computed from the restriction of the vanishing cycle complex to only singular strata with a certain lower bound in dimension. The same method applies to the study of the monodromy eigenspaces of Milnor fiber cohomology; this is discussed in \S\ref{eigsh}. In \S\ref{Mm}, we discuss divisibility results for the characteristic polynomials of Milnor monodromy, and upper bounds for the maximal size of Jordan blocks.

  
 \subsection{Betti bounds via perverse vanishing cycles}\label{bound} \ 
 In this section, we derive information on the Milnor fiber cohomology as a consequence of the perversity of vanishing cycles (compare also with \cite{DS}).
 
Under the notations and assumptions of \S\ref{nv}\footnote{In fact, in Section \ref{bound} we use only the perversity in the usual sense of $\underline{\bZ}_X[n+1]$, except for part (d) of Theorem \ref{c14a} where the strong perversity assumption is needed.}, we consider the stratified singular locus $\Sigma:={\rm Sing}_{\cW}(f)$, which is  defined as the union of the singular 
loci $\Sing f_{|W_i}$ of the restrictions of $f$ to the strata  $W_i$ of fixed Whitney stratification $\cW$. 
Let us remark that $\Sigma$  is a closed set  and a subset of $Y$. We may and will assume in the following that 
it is a union of strata in $\cW$. Let $s:= \dim_0 \Sigma.$

Let $B_\epsilon$ be a Milnor ball for $f$ at the origin, that is, the intersection of a small enough ball at the origin of the ambient space $\bC^N$ with a suitable representative of the germ $(X, 0)$. Let $$F:=B_\epsilon \cap f^{-1}(\gamma),  \ \ 0<\gamma \ll \epsilon,$$ 
be the Milnor fiber of $f$ at the origin.
 
Assume $0<s<n$ and let 
$$\cFb:={^p\varphi_f}(\underline{\bZ}_X[n+1]),$$
where $^p\varphi_f:=\varphi_f[-1]$ denotes as in \S\ref{nv} the perverse vanishing cycle functor associated to $f$.
 Since $\underline{\bZ}_X[n+1]$ was assumed to be a perverse sheaf on $X$, we get that $\cFb$ is perverse on $Y=f^{-1}(0)$. Moreover, since $\cFb$ is supported on $\Sigma$, we obtain that $$\cFb_0:=\cFb\vert_\Sigma$$ is a perverse sheaf on $\Sigma$ (e.g., see \cite[Corollary 8.2.10]{Max}). 
For any integer $k$ we have:
\be\label{ad1}\widetilde{H}^k(F) \cong \cH^{k-n}(\cFb)_0 \cong \cH^{k-n}(\cFb_0)_0.\ee
The support condition for the perverse sheaf $\cFb_0$ on $\Sigma$ then yields (e.g., as in \cite[Proposition 10.6.2]{Max}) that the only possibly non-trivial integral reduced cohomology $\widetilde{H}^k(F)$ of $F$ is concentrated in degrees $n-s \leq k \leq n$. 

The computation of $\widetilde{H}^k(F)$ is equivalent via \eqref{ad1} to the computation of  the hypercohomology group $\bH^{k-n}(B_\epsilon \cap \Sigma;\cFb_0)$, where $B_\epsilon$ is as above a Milnor ball for $f$ at the origin. For convenience of notation, we will follow the convention mentioned in the introduction, namely {\it we will assume throughout the paper that we work in a sufficiently small open ball $B_\epsilon$ around the origin, and replace $B_\epsilon \cap \Sigma$ with simply $\Sigma$, $B_\epsilon \cap Y$ with $Y$, $B_\epsilon$ with $X$, etc.} In particular, we shall write
\be\label{ad1e}
\widetilde{H}^k(F) \cong \bH^{k-n}(\Sigma;\cFb_0),
\ee
with $\bH^*$ denoting the hypercohomology functor.

Denote by $\cS$ a Whitney stratification of $Y$ so that $\cFb$ is $\cS$-constructible. Upon refining $\cW$, we can assume that any stratum in $\cS$ is also a stratum in $\cW$. Since we work in a neighborhood of the origin, we can further assume (after shrinking $X$, and restricting to a normal slice through the origin as in Example \ref{rem010} below) that the origin is the only zero-dimensional stratum of $\cS$ and $\Sigma$ is the union of strata of complex dimension $\leq s$. Let us make the following notations:
\begin{itemize}
\item[] $\Sigma_{\ell}=$ the union of $\ell$-dimensional strata of $\cS$ (so $\Sigma_0=\{0\}$);
\item[] $U_{\ell}=$ the union of strata of $\cS$ of complex dimension $\geq \ell$.
\end{itemize}
Then each $U_{\ell}$ is an open subset of $\Sigma$, and 
\be\label{of}\Sigma_s=U_s \subseteq U_{s-1} \subseteq \cdots \subseteq U_0=\Sigma\ee
with $\Sigma_{\ell} = U_{\ell} \setminus U_{\ell+1}$. 
Set $$\cFb_\ell:=\cFb_0\vert_{U_\ell}$$ and note that $\cFb_\ell$ is a perverse sheaf of $U_\ell$. 

We are now ready to prove the following result, which can be seen as an enhancement of results from \cite{DS}. However, note that while \cite{DS} works with $\bQ$-coefficients and the restriction of the vanishing cycle complex to the real link (see Remark \ref{dif}(b) below), our result holds over the integers and is formulated in terms of the germs of strata of the stratified singular locus of $f$.
\bt\label{t14a}
Let $f:(X,0) \to (\bC,0)$ be a nonconstant holomorphic function defined on a pure $(n+1)$-dimensional complex singularity germ satisfying the property that  
$\underline{\bZ}_X[n+1]$ is a $\bZ$-perverse sheaf on $X$.
Assume that the stratified singular locus $\Sigma$ of $f$ is of complex dimension $s>0$. If $F$ denotes the Milnor fiber of $f$ at the origin, then using the above notations for the stratification of $\Sigma$, the following hold:
\begin{itemize}
\item[(a)]  for any $j=0,\ldots, s-1$ there is a monomorphism
\be\label{newthad}
\widetilde{H}^{n-s+j}(F) \hookrightarrow \bH^{-s+j}(U_{s-j};\cFb_{s-j}).
\ee
\item[(b)]  if $s\geq 2$, then for any $j=0,\ldots,s-2$ there is an isomorphism
\be\label{newthad2}
\widetilde{H}^{n-s+j}(F) \cong \bH^{-s+j}(U_{s-j-1};\cFb_{s-j-1}).
\ee
\end{itemize}
\et

\begin{proof} As already mentioned above, we assume that the singularity germ $X$ is represented by its intersection with a sufficiently small open ball $B_\epsilon$ at the origin of $\bC^N$.

(a) \  Fix an integer $j=0,\ldots, s-1$.  In the notations preceding the theorem, 
for any $0 \leq \ell \leq s-j-1$ consider the inclusions
$$\Sigma_\ell \overset{v_\ell}{\hookrightarrow} U_\ell \overset{u_\ell}{\hookleftarrow} U_{\ell+1}$$
and the attaching distinguished triangle 
\be\label{stl}{v_\ell}_!{v_\ell}^!\cFb_\ell \to \cFb_\ell \to R{u_\ell}_*{u_\ell}^*\cFb_\ell \overset{[1]}{\to}\ee
with ${v_\ell}_!={v_\ell}_*$ and ${u_\ell}^*\cFb_\ell \cong \cFb_{\ell+1}$. The hypercohomology long exact sequence associated to \eqref{stl} contains the terms
\be\label{stl2} \cdots \to \bH^{-s+j}(\Sigma_\ell; {v_\ell}^!\cFb_\ell) \to \bH^{-s+j}(U_\ell;\cFb_\ell) \to   \bH^{-s+j}(U_{\ell+1};\cFb_{\ell+1}) \to \cdots\ee
The group $\bH^{-s+j}(\Sigma_\ell; {v_\ell}^!\cFb_\ell)$ is computed by a hypercohomology spectral sequence whose $E_2$-term is given by
\be\label{spsup} E_2^{p,q}=H^p(\Sigma_\ell; \cH^q({v_\ell}^!\cFb_\ell)).\ee
Since ${v_\ell}^!\cFb_\ell$ is a constructible complex on $\Sigma_\ell$, its cohomology sheaves are local systems on every connected component of $\Sigma_\ell$. Hence, by reasons of dimension, 
$E_2^{p,q}=0$ if $p<0$ or $p>2\ell$. On the other hand, the costalk condition \eqref{coss} for the perverse sheaf $\cFb_\ell$ on $U_\ell$ (with the induced stratification) yields that $\cH^q({v_\ell}^!\cFb_\ell) \cong 0$ for all $q < -\ell$. Therefore, $E_2^{p,q}=0$ if $q<-\ell$. 
Altogether, since $-s+j<-\ell$, we get that $E_2^{p,q}=0$ for any pair of integers $(p,q)$ with $p+q=-s+j$. The spectral sequence \eqref{spsup} then implies that 
\be\label{stl3}\bH^{-s+j}(\Sigma_\ell; {v_\ell}^!\cFb_\ell)\cong 0.\ee
By combining \eqref{stl2} and \eqref{stl3}, we get a monomorphism
\be \bH^{-s+j}(U_\ell;\cFb_\ell) \hookrightarrow   \bH^{-s+j}(U_{\ell+1};\cFb_{\ell+1})\ee for any $0 \leq \ell \leq s-j-1$.
Together with \eqref{ad1e} and noting that $\Sigma=U_0$, the above discussion yields a composition of monomorphisms
$$\widetilde{H}^{n-s+j}(F) \cong \bH^{-s+j}(U_0;\cFb_0) \hookrightarrow  \bH^{-s+j}(U_1;\cFb_1) \hookrightarrow \cdots \hookrightarrow  \bH^{-s+j}(U_{s-j};\cFb_{s-j}),$$
thus completing the proof of \eqref{newthad}.

 (b) \  Let us now assume that $s \geq 2$ and fix an integer $j=1,\ldots,s-1$. 

In view of \eqref{stl3}, the long exact sequence \eqref{stl2} contains the terms:
\be\label{stl2b} \cdots \to \bH^{-s+j-1}(\Sigma_\ell; {v_\ell}^!\cFb_\ell) \to \bH^{-s+j-1}(U_\ell;\cFb_\ell) \to   \bH^{-s+j-1}(U_{\ell+1};\cFb_{\ell+1}) \to 0 \to \cdots\ee
For any $0 \leq \ell \leq s-j-1$, the same arguments used for studying the spectral sequence \eqref{spsup} yield that $\bH^{-s+j-1}(\Sigma_\ell; {v_\ell}^!\cFb_\ell) \cong 0$ since $-s+j-1<-\ell$. In particular, \eqref{stl2b} yields isomorphisms
$$\bH^{-s+j-1}(U_\ell;\cFb_\ell) \cong   \bH^{-s+j-1}(U_{\ell+1};\cFb_{\ell+1}) $$
for all $0 \leq \ell \leq s-j-1$. Together with \eqref{ad1e}, this then yields isomorphisms
\be\label{rei}\widetilde{H}^{n-s+j-1}(F) \cong \bH^{-s+j-1}(U_0;\cFb_0) \cong  \bH^{-s+j-1}(U_1;\cFb_1) \cong \cdots \cong  \bH^{-s+j-1}(U_{s-j};\cFb_{s-j}).\ee
The isomorphism \eqref{newthad2} is then obtained by reindexing (i.e., replacing $j$ by $j+1$ in \eqref{rei}).
\end{proof}

An immediate consequence of Theorem \ref{t14a}(b) is the following.  
\begin{corollary}\label{c2.2}
If $s \geq 2$, then for any $j=0,\ldots, s-2$, the group $\widetilde{H}^{n-s+j}(F;\bZ)$ depends only on the singular strata of dimension $\geq s-j-1$ of $\Sing_\cW(f)$.
\end{corollary}

\begin{remark}\label{dif} \ \\
(a) Assuming $s\geq 2$ and fixing  $j=0,\ldots, s-2$, if there are no strata of dimension $s-j-1$, then $U_{s-j}=U_{s-j-1}$, so in this case \eqref{newthad2} is a finer result than \eqref{newthad}. In general, the right-hand side of either \eqref{newthad} or \eqref{newthad2} can be computed via the hypercohomology spectral sequence, though explicit computations can be tedious.\\
(b) The results of Theorem \ref{t14a} and Corollary \ref{c2.2} show that $\widetilde{H}^{j}(F)$ for $j<n-1$ (resp., $j=n-1$) is {completely} (resp., partially) determined by the restriction of the vanishing cycle complex $\varphi_f\underline{\bZ}_X$ to the complement of the singular point at the origin. A similar statement, though not as explicit as Corollary \ref{c2.2},  can be derived from  \cite[Theorem 0.1]{DS}, where one considers the restriction of the vanishing cycle complex to the real link of the singularity. Specifically, if $K$ denotes the real link of $0 \in Y=f^{-1}(0)$, i.e., the intersection of $Y$ with a sufficiently small sphere around $0$ in a smooth ambient space $\bC^N$, then under our assumptions and notations one gets isomorphisms
\be\label{ds1} \widetilde{H}^k(F;\bZ) \cong \bH^k(K;\varphi_f \underline{\bZ}_X\vert_{K}) 
\ee
for all  $k<n-1$,
and a monomorphism
\be\label{ds2} \widetilde{H}^{n-1}(F;\bZ) \hookrightarrow \bH^{n-1}(K;\varphi_f \underline{\bZ}_X\vert_{K}).\ee
For the benefit of the reader, we include here the elementary proof of \eqref{ds1} and \eqref{ds2}. Recall that $X$ is represented by its intersection with a Milnor ball $B_\epsilon$ at the origin. Denote by 
$i:\{0\} \hookrightarrow Y$ and $j:Y \setminus \{0\} \hookrightarrow Y$ the inclusion maps. Let $\cFb$ be a $\bZ$-perverse sheaf on $Y$. The costalk conditions for $\cFb$ yields the vanishing:
\begin{itemize}
\item[(i)] $H^k(i^!\cFb) = 0$ if $k<0$.
\end{itemize}
Moreover, it is well known that
\begin{itemize}
\item[(ii)] $H^k(i^*Rj_*j^*\cFb) \cong \bH^k(K;\cFb\vert_{K})$, where $K$ is the real link of 
$0 \in Y$.
\end{itemize}
By applying the pullback $i^*$ to the attaching triangle 
$$ i_!i^!\cFb \to \cFb \to Rj_*j^* \cFb \overset{[1]}{\to}$$
and using the fact that $i^*i_!\simeq id$, one gets the distinguished triangle 
\be\label{dtads} 
i^!\cFb \to i^*\cFb \to i^*Rj_*j^* \cFb \overset{[1]}{\to}
\ee
In view of (i) and (ii), the long exact sequence of hypercohomology groups associated to \eqref{dtads} yields isomorphisms
\be\label{l1} H^k(i^*\cFb) \cong \bH^k(K;\cFb\vert_{K}) , \ \ \forall \ k<-1,
\ee
and a monomorphism
\be\label{l2} 
H^{-1}(i^*\cFb) \hookrightarrow \bH^{-1}(K;\cFb\vert_{K}).
\ee
To obtain \eqref{ds1} and \eqref{ds2}, one simply applies \eqref{l1} and \eqref{l2} to the perverse vanishing cycles $\cFb:={^p\varphi_f}(\underline{\bZ}_X[n+1])$.
It should be noted that Dimca and Saito worked in \cite{DS} with $\bQ$-coefficients, but as seen above their result extends easily to the integers. 

Finally, let us compare the above arguments with the statement and proof of our Theorem  \ref{t14a}. It is well known that complex analytic sets are {\it locally conelike} (see \cite{BV}), hence $U_1=\Sigma \setminus \{0\}$ is stratwise topologically equivalent with the product of the link of $0$ in $\Sigma$ (that is, $K \cap \Sigma$) with an open interval $(0,\epsilon)$. Therefore, in the special case $s-j=2$, the isomorphism \eqref{newthad2}  reproves \eqref{ds1} with $k=n-2$. In all other cases, our isomorphism \eqref{newthad2}  is strictly finer than \eqref{ds1}. Moreover, the monomorphism \eqref{newthad} specializes, after setting $s-j=1$, to \eqref{ds2}, while none of the other cases addressed by \eqref{newthad} (except at $j=0$) has a counterpart in \cite{DS}. This justifies our assertion that Theorem \ref{t14a} provides a new perspective and an enhancement of some of the results from \cite{DS}.
\\
(c) If $f$ is a homogeneous polynomial, a more refined dependence of the vanishing cohomology on the singular strata was obtained in \cite[Proposition 5.1]{Max1} (see also \cite[Theorem 3.1]{Li} for the case when $f$ defines a central hyperplane arrangement).
\end{remark}

\medskip

In what follows, we specialize Theorem \ref{t14a} to the case $j=0$, to derive more explicit information about $\widetilde{H}^{n-s}(F)$, that is, the lowest (possibly non-trivial) cohomology group of $F$ (compare also with \cite[Section 3.5]{DS} for a related discussion). Just like in Theorem \ref{t14a}, most arguments used in the proof of Theorem \ref{c14a} below use only the perversity in the usual sense of $\cFb$. But the statement about the freeness of 
$\widetilde{H}^{n-s}(F)$ (part (d) of Theorem \ref{c14a} ) requires the strong perversity of $\cFb$, deduced via Proposition \ref{p28} from the assumption $\r(X,\bZ)=n+1$.
We first introduce some notations.

Recall that $$U_s=\Sigma_s=\bigsqcup_i \Sigma_{s,i},$$
where $\Sigma_{s,i}$ are the $s$-dimensional (connected) strata of $\Sigma$. Denote by $F^\pitchfork_{s,i}$ the Milnor fiber of $f$ at a point $x_{s,i} \in \Sigma_{s,i}$. Let $N$ be a normal slice to the stratum $\Sigma_{s,i}$ at the point  $x_{s,i}$ (i.e., a smooth analytic subvariety of $\bC^N$, transversal to $\Sigma_{s,i}$ at $x_{s,i}$).
By the base change  formula for vanishing cycles (e.g., see \cite[Lemma 4.3.4]{Sc}), we get
\be\label{1000} \widetilde{H}^k(F^\pitchfork_{s,i}) \cong \cH^{k}(\varphi_f\underline{\bZ}_X)_{x_{s,i} } \cong \cH^{k}(\varphi_{f\vert_N} \underline{\bZ}_{X\cap N})_{x_{s,i} },\ee
i.e.,  $\widetilde{H}^k(F^\pitchfork_{s,i})$  can be identified with the $k$-th reduced cohomology of the Milnor fiber of the restriction $f\vert_N$ of $f$ to the normal slice $N$ to the stratum $\Sigma_{s,i}$. For this reason, we will simply refer to $F^\pitchfork_{s,i}$ as the {\it transversal Milnor fiber} of $f$ along $\Sigma_{s,i}$. Furthermore, by transversality, the function $f\vert_N$ has an {\it isolated} (stratified) singularity at $x_{s,i}$, and hence the point $x_{s,i}$ is an isolated point in the support of the 
perverse sheaf $\varphi_{f\vert_N} \underline{\bZ}_{X\cap N}[n-s]$. Therefore, the stalk 
cohomology of $\varphi_{f\vert_N} \underline{\bZ}_{X\cap N}[n-s]$ at $x_{s,i}$ 
is concentrated in degree $0$. 
This then implies that $\widetilde{H}^k(F^\pitchfork_{s,i})$ is trivial except possibly in degree $k=n-s$. 
Denote by $\mu^\pitchfork_{s,i}$ the rank of $\widetilde{H}^{n-s}(F^\pitchfork_{s,i})$.

For any $i$, the fundamental group $\pi_1(\Sigma_{s,i})$ of the stratum $\Sigma_{s,i}$ acts on $\widetilde{H}^{n-s}(F^\pitchfork_{s,i})$ via a  homomorphism 
\be\label{vm}
 A_i:\pi_1(\Sigma_{s,i}) \lra \Aut\big(\widetilde{H}^{n-s}(F^\pitchfork_{s,i})\big),
 \ee
which determines a local system $\mathcal{L}_{s,i}$ on $\Sigma_{s,i}$, with stalk $\widetilde{H}^{n-s}(F^\pitchfork_{s,i})$. 
We refer to $A_i$ as the {\it local system monodromy} along the stratum $\Sigma_{s,i}$.
If $\pi_1(\Sigma_{s,i}) \cong \bZ$ (e.g., if $s=1$), the homomorphism $A_i$ can be regarded as an automorphism $A_i: \widetilde{H}^{n-s}(F^\pitchfork_{s,i}) \to \widetilde{H}^{n-s}(F^\pitchfork_{s,i})$; in this case, it will be referred to as the {\it vertical monodromy} along $\Sigma_{s,i}$, by analogy with the case $s=1$ considered in \cite{Si3}.

Let us also denote by $\{\Sigma_{s-1,j}\}_j$ the collection of connected singular strata of dimension $s-1$, and for each $j$ we fix a point $x_j \in \Sigma_{s-1,j}$.

With the above notations and under the hypotheses of Theorem \ref{t14a}, we can now prove the following.\footnote{M. Saito communicated to us that parts (a)-(c) of Theorem \ref{c14a} can also be deduced from arguments {\it implicit}  in \cite{DS}.}
\begin{theorem}\label{c14a}
\begin{enumerate}
\item There is a monomorphism
\be\label{upb}
\widetilde{H}^{n-s}(F) \hookrightarrow \bigoplus_i \widetilde{H}^{n-s}(F^\pitchfork_{s,i})^{A_i},
\ee
where $F^\pitchfork_{s,i}$ is the transversal Milnor fiber to the $s$-dimensional stratum $\Sigma_{s,i}$, and $A_i$ is the local system monodromy along $\Sigma_{s,i}$.\footnote{Since we are interested only in  the $A_i$-invariant part of $\widetilde{H}^{n-s}(F^\pitchfork_{s,i})$, it is clear that this is equal to the intersection of the invariant submodules over some set of generators of $\pi_1(\Sigma_{s,i})$. See also \eqref{eq:hatA} and the conjecture after it.}

In particular, taking ranks yields the inequalities:
 \be\label{2.18}
\widetilde{b}_{n-s}(F) \leq \sum_i \rank \widetilde{H}^{n-s}(F^\pitchfork_{s,i})^{A_i} \leq \sum_i  \mu^\pitchfork_{s,i}.
 \ee
\item If, moreover, $s\geq 2$, then
 \be\label{lpb}
\widetilde{b}_{n-s}(F) \geq \sum_i  \rank \widetilde{H}^{n-s}(F^\pitchfork_{s,i})^{A_i} - \sum_j \rank H^{s-1}(i_{x_j}^!\cFb) ^{\pi_1(\Sigma_{s-1,j})},
 \ee
 where $x_j$ is some point in the $(s-1)$-dimensional stratum $\Sigma_{s-1,j}$. \\
\item If $s\geq 2$ and the germ of the singular locus $\Sigma$ at the origin has no strata of dimension $s-1$, then \eqref{upb} is an isomorphism and  the first inequality in \eqref{2.18} becomes an equality. 
\item If $\r(X,\bZ)=n+1$, then $\widetilde{H}^{n-s}(F)$ is free.

\end{enumerate}
\end{theorem}

\begin{proof}
(a) We continue to use here the notations from Theorem \ref{t14a}. 
First note that \eqref{newthad} yields a monomorphism
$$
\widetilde{H}^{n-s}(F) \hookrightarrow \bH^{-s}(\Sigma_{s};\cFb_{s}).
$$
The hypercohomology spectral sequence together with the support condition for perverse sheaves then yields that (e.g., see \cite[Proposition 5.2.20]{Di})
$$\bH^{-s}(\Sigma_s;\cFb_s)\cong H^0(\Sigma_s;\cH^{-s}(\cFb_s))
\cong \bigoplus_i H^{0}(\Sigma_{s,i};\cH^{-s}(\cFb_s)\vert_{\Sigma_{s,i}}).$$
By constructibility, $\cH^{-s}(\cFb_s)\vert_{\Sigma_{s,i}}$ is a local system on $\Sigma_{s,i}$ with stalk $\widetilde{H}^{n-s}(F^\pitchfork_{s,i})$, which in our previous notations is exactly 
$\mathcal{L}_{s,i}$. Finally, it is well known (e.g., see \cite[Exercise 4.2.16]{Max}) that 
$$H^{0}(\Sigma_{s,i};\mathcal{L}_{s,i}) \cong \widetilde{H}^{n-s}(F^\pitchfork_{s,i})^{A_i},$$
with the right-hand side denoting the fixed part of $\widetilde{H}^{n-s}(F^\pitchfork_{s,i})$ under the $A_i$-action. Altogether, this proves \eqref{upb}, while \eqref{2.18} follows by computing ranks in \eqref{upb}.

\noindent (b) Let us next assume that $s\geq 2$. By setting $j=0$ in \eqref{newthad2} we get an isomorphism
\be\label{usm1}
\widetilde{H}^{n-s}(F) \cong \bH^{-s}(U_{s-1};\cFb_{s-1}).
\ee
Consider the inclusions
$$\Sigma_{s-1} \overset{\alpha}{\hookrightarrow} U_{s-1} \overset{\beta}{\hookleftarrow} U_{s}=\Sigma_s$$
(i.e., in the notations of Theorem \ref{t14a}, $\alpha=v_{s-1}$ and $\beta=u_{s-1}$)
and the corresponding attaching triangle 
\be\label{stls1}{\alpha}_!{\alpha}^!\cFb_{s-1} \to \cFb_{s-1} \to R{\beta}_*{\beta}^*\cFb_{s-1} \overset{[1]}{\to}\ee
with ${\alpha}_!={\alpha}_*$ and ${\beta}^*\cFb_{s-1} \cong \cFb_{s}$. In view of \eqref{usm1}, the hypercohomology long exact sequence associated to \eqref{stls1} contains the terms
\be\label{stl2s} \cdots \to \bH^{-s}(\Sigma_{s-1}; {\alpha}^!\cFb_{s-1}) \to \widetilde{H}^{n-s}(F)  \to   \bH^{-s}(U_{s};\cFb_{s}) \to \bH^{-s+1}(\Sigma_{s-1}; {\alpha}^!\cFb_{s-1}) \to \cdots\ee
As in the proof of Theorem \ref{t14a}, a spectral sequence computation yields that $$\bH^{-s}(\Sigma_{s-1}; {\alpha}^!\cFb_{s-1}) \cong 0,$$ and, as in the proof of \eqref{upb} above, we have:
$$\bH^{-s}(U_{s};\cFb_{s})\cong  \bigoplus_i \widetilde{H}^{n-s}(F^\pitchfork_{s,i})^{A_i}.$$
Therefore, \eqref{stl2s} yields the exact sequence
\be\label{stl3s} 0 \to \widetilde{H}^{n-s}(F)  \to \widetilde{H}^{n-s}(F^\pitchfork_{s,i})^{A_i} \to \bH^{-s+1}(\Sigma_{s-1}; {\alpha}^!\cFb_{s-1}) \to \cdots\ee
Since ${\alpha}^!\cFb_{s-1}$ is a constructible complex on $\Sigma_{s-1}$, its cohomology sheaves are local systems on every connected component $\Sigma_{s-1,j}$ of $\Sigma_{s-1}$. Moreover, the hypercohomology spectral sequence yields that
\begin{equation}\label{k1}
\begin{split} \bH^{-s+1}(\Sigma_{s-1}; {\alpha}^!\cFb_{s-1}) &\cong H^0(\Sigma_{s-1}; \cH^{-s+1}({\alpha}^!\cFb_{s-1})) \\ &\cong \bigoplus_j H^0(\Sigma_{s-1,j}; \cH^{-s+1}({\alpha}^!\cFb_{s-1})\vert_{\Sigma_{s-1,j}}).
\end{split}
\end{equation}
If $x \in \Sigma_{s-1}$, with inclusions 
$$\{x\} \overset{k_x}{\hookrightarrow} \Sigma_{s-1} \overset{\alpha}{\hookrightarrow} U_{s-1}$$
and $i'_x:=\alpha \circ k_x: \{x\} \hookrightarrow U_{s-1}$, then we have as in \cite[Remark 6.0.2(1)]{Sc} that
$$k_x^*\alpha^!\cong k_x^!\alpha^![2(s-1)] \cong (i'_x)^![2(s-1)].$$ In particular, the stalk of $\cH^{-s+1}({\alpha}^!\cFb_{s-1})$ at a point $x \in \Sigma_{s-1}$ is computed by:
 \be\label{k2} \cH^{-s+1}({\alpha}^!\cFb_{s-1})_x=H^{-s+1}(k_x^*\alpha^!\cFb_{s-1})\cong H^{s-1}\big((i'_x)^!\cFb_{s-1})\big.\ee
 Furthermore, if $i_x:\{x\} \hookrightarrow Y$ is the inclusion map, then it follows (e.g., using the proof of \cite[Corollary 8.2.10]{Max}) that \be\label{225}(i'_x)^!\cFb_{s-1} \cong i_x^!\cFb.\ee
 Therefore, after choosing a point $x_j$ in each $(s-1)$-dimensional stratum $\Sigma_{s-1,j}$, we get from \eqref{k1}, \eqref{k2} and \eqref{225} that 
 \be\label{2:27}
 \bH^{-s+1}(\Sigma_{s-1}; {\alpha}^!\cFb_{s-1}) \cong \bigoplus_j  H^{s-1}(i_{x_j}^!\cFb) ^{\pi_1(\Sigma_{s-1,j})}.
 \ee
The inequality \eqref{lpb} follows now by taking ranks in \eqref{stl3s} and using \eqref{2:27}.

The assertion in (c) follows  immediatly from \eqref{stl3s}. 

To prove (d), we deduce from \eqref{upb} that it suffices to show that $\widetilde{H}^{n-s}(F^\pitchfork_{s,i})$ is free, for each $i$. In the notations of \eqref{1000}, we have:
$$\widetilde{H}^{n-s}(F^\pitchfork_{s,i}) \cong \cH^{n-s}(\varphi_{f\vert_N} \underline{\bZ}_{X\cap N})_{x_{s,i}} \cong \cH^{0}(\varphi_{f\vert_N} \underline{\bZ}_{X\cap N}[n-s])_{x_{s,i}}.$$
Since the point $x_{s,i}$ is an isolated point in the support of the 
perverse sheaf $\varphi_{f\vert_N} \underline{\bZ}_{X\cap N}[n-s]$, the stalk 
and costalk 
cohomology of $\varphi_{f\vert_N} \underline{\bZ}_{X\cap N}[n-s]$ at $x_{s,i}$ 
are isomorphic, and they are concentrated in degree $0$. It is therefore enough to show that $\varphi_{f\vert_N} \underline{\bZ}_{X\cap N}[n-s]$ is a strongly perverse sheaf.

For this, we first note that the assumption $\r(X,\bZ)=n+1$ implies via Proposition \ref{p28} that $\underline{\bZ}_{X}[n+1]$ is strongly perverse. Then the fact that $N$ is transversal to the stratification of $X$ can be used to show that $\underline{\bZ}_{X\cap N}[n+1-s]$ is strongly perverse on $X \cap N$. Indeed, by transversality, $X\cap N$ gets an induced stratification with strata of the form $S \cap N$ for $S$ a stratum in $X$. 
For a stratum $S$  of $X$ and $x \in S \cap N$, we denote by $k'_x:\{x\} \hookrightarrow X \cap N$ and $k_x:\{x\} \hookrightarrow X$ the point inclusions. If we let $i_N:X \cap N \hookrightarrow X$ be the inclusion map,  then $i_N \circ k'_x=k_x$,  and $i_N^!=i_N^*[-2s]$ (see, e.g., \cite[(6.44)]{Sc}). The fact that $\underline{\bZ}_{X\cap N}[n+1-s]$ is perverse on $X \cap N$ follows immediately from the perversity of $\underline{\bZ}_{X}[n+1]$ (see, e.g., \cite[Lemma 6.0.4]{Sc}). 
Moreover, 
\begin{equation*} \begin{split} H^{\dim S}(k_x^! \underline{\bZ}_{X}[n+1]) &\cong  H^{\dim S}\big((k'_x)^! i_N^! \underline{\bZ}_{X}[n+1]\big) \cong H^{\dim S}\big((k'_x)^! i_N^* \underline{\bZ}_{X}[n+1-2s]\big) \\ &\cong H^{\dim S-s}\big((k'_x)^!  \underline{\bZ}_{X\cap N}[n+1-s]\big) ,\end{split}\end{equation*}
which translates into the fact that the perverse sheaf $\underline{\bZ}_{X\cap N}[n+1-s]$ on $X \cap N$ is strongly perverse.
Finally, the stability of strong perversity under the perverse vanishing cycle functor $\varphi_{f\vert_N}[-1]$ implies that  $\varphi_{f\vert_N} \underline{\bZ}_{X\cap N}[n-s]$ is a strongly perverse sheaf, as desired.\footnote{Similar arguments regarding freeness are used by the authors in the proof of \cite[Theorem 1.2]{MPT}, in the context of vanishing cohomology of complex projective hypersurfaces.} 

Let us point out that the freeness is also proved in Theorem \ref{t:euler}(a).
\end{proof}

\begin{remark}\label{r65}
The exact sequence \eqref{stl3s} shows that the ``correction'' of \eqref{upb} from being an isomorphism depends only on the $(s-1)$-dimensional strata of the singular locus $\Sigma$ and of the costalks of the perverse vanishing cycle complex $\cFb$ at points in these strata. A similar remark also follows from Theorem \ref{t:euler} after reducing to $s=2$ by iterated slicing.
As we will see in some of the examples in \S \ref{examples}, the monomorphism \eqref{upb} is not in general an isomorphism if $(s-1)$-dimensional strata are present in $\Sigma$. 
\end{remark}

\begin{remark}
If we work with $\bQ$-coefficients, the statements and proofs of Theorems \ref{t14a} and  \ref{c14a} hold in the category of mixed Hodge modules, provided that $\underline{\bQ}_X$ exists in the derived category of mixed Hodge modules (e.g., if $X$ is a complete intersection in a complex manifold, cf. \cite[Proposition 2.19]{Sa90}). In particular,  under this assumption, \eqref{newthad} is a monomorphism in the category of mixed Hodge structures and \eqref{newthad2}  is an isomorphism of mixed Hodge structures. 
\end{remark}

\subsection{Characteristic polynomials of Milnor monodromy. Jordan blocks.}\label{Mm} \
In this section, cohomology groups are taken with $\bC$-coefficients. Let $h$ and $h_{i}$ denote the Milnor monodromy on the cohomology of $F$, and on the transversal Milnor fiber $F_{s,i}^{\pitchfork}$ to some $s$-dimensional stratum, respectively. Let $\mathrm{char}_{h |\widetilde{H}^{n-s}(F)}$  denote the characteristic polynomial of the monodromy $h$ acting on $\widetilde{H}^{n-s}(F)$.
Let $b_\lambda(V, \mu)$ denote the dimension of the eigenspace
corresponding to the eigenvalue $\lambda$ of the linear operator $\mu$ acting on the vector space
$V$, and let  $J_\lambda(V, \mu)$ denote the maximum of the sizes of the 
Jordan blocks.

Since the monomorphism \eqref{upb} is compatible with the Milnor monodromy actions, with these notations we get the following.
 
 \begin{corollary}\label{c:2}\label{c:3}
The characteristic polynomial  $\mathrm{char}_{h^|\widetilde{H}^{n-s}(F)}$ divides $\prod_{i} \mathrm{char}_{h_i|\widetilde{H}^{n-s}(F^\pitchfork_{s,i})^{A_i}}$. 
 In particular,  $\mathrm{char}_{h^|\widetilde{H}^{n-s}(F)}$ divides the product
 $\prod_{i} \mathrm{char}_{h_i|\widetilde{H}^{n-s}(F^\pitchfork_{s,i})}$.  
 
 Moreover, we have the inequalities:
 \begin{enumerate}
\item[(i)] $b_\lambda(\widetilde{H}^{n-s}(F), h) \le \sum_{i} b_\lambda(\widetilde{H}^{n-s}(F_{s,i}^{\pitchfork})^{A_i}, h_i)$.
\item[(ii)] $J_\lambda(\widetilde{H}^{n-s}(F), h) \le \sum_{i} J_\lambda(\widetilde{H}^{n-s}(F_{s,i}^{\pitchfork})^{A_i}, h_i)$.
\end{enumerate} 
\fin
\end{corollary}

Corollary \ref{c:2} extends the results  \cite[Corollaries 3.2 and 3.3]{Ti-nonisol}; compare also with \cite[Section 3.3]{DS}.


 \subsection{Eigenspaces of monodromy.}\label{eigsh} 
The proofs of Theorems \ref{t14a} and \ref{c14a} apply to any perverse sheaf supported on the stratified singular locus of $f$, e.g., to the generalized eigensheaves $ \varphi_{f,\lambda}\underline{\bC}_X$ ($\lambda \in \bC$) of vanishing cycles. 
With the notations of \S\ref{nv}, consider the $\bC$-perverse sheaf $${\cFb^\lambda}:={^p\varphi_{f,\lambda}}(\underline{\bC}_X[n+1])$$ on $Y$, and let $\Sigma^\lambda$ be the support of $\cFb^\lambda$ with $s_\lambda=\dim \Sigma^\lambda$. Then $$\cFb^\lambda \vert_{\Sigma^\lambda}=:\cFb^\lambda_0$$ is a perverse sheaf on $\Sigma^\lambda$ and we have the isomorphisms:
\be\label{ad1l}
\widetilde{H}^k(F;\bC)_\lambda \cong \cH^{k-n}(\cFb^\lambda)_0 \cong \cH^{k-n}(\cFb^\lambda_0)_0 \cong \bH^{k-n}(\Sigma^\lambda;\cFb^\lambda_0).
\ee
The support condition for perverse sheaves then yields immediately that
\be
\widetilde{H}^{n-s_\lambda-j}(F;\bC)_\lambda=0
\ee
for all $j>0$, so the only interesting $\lambda$-eigenspaces of Milnor monodromy are $\widetilde{H}^{k}(F;\bC)_\lambda$ with $k=n-s_\lambda, \ldots, n$.
Let us next note that $\Sigma^\lambda \subseteq \Sigma$ is a closed union of strata in the Whitney stratification of $Y$, and it is exhausted by opens as in \eqref{of}, where we use the upperscript $\lambda$ when considering strata which are contained in $\Sigma^\lambda$.
Let us set 
$$\cFb^\lambda_\ell:=\cFb^\lambda_0\vert_{U^\lambda_\ell}$$ and note that $\cFb^\lambda_\ell$ is a perverse sheaf of $U^\lambda_\ell$ (i.e., on the union of strata in $\Sigma^\lambda$ of complex dimension $\geq \ell$). 
With these notations, the proof of Theorem \ref{t14a} adapted to the perverse sheaf $\cFb^\lambda_0$ on $\Sigma^\lambda$ yields the following (for $j=0$, compare also with \cite[Section 3.5]{DS}).
\bt\label{t15a}
\begin{itemize}
\item[(a)]  For any $j=0,\ldots, s_\lambda-1$ there is a monomorphism
\be\label{newthadb}
\widetilde{H}^{n-s_\lambda+j}(F;\bC)_\lambda \hookrightarrow \bH^{-s_\lambda+j}(U^\lambda_{s_\lambda-j};\cFb^\lambda_{s_\lambda-j}).
\ee
\item[(b)]  If $s_\lambda\geq 2$, then for any $j=0,\ldots,s_\lambda-2$ there is an isomorphism
\be\label{newthad2b}
\widetilde{H}^{n-s_\lambda+j}(F;\bC)_\lambda \cong \bH^{-s_\lambda+j}(U^\lambda_{s_\lambda-j-1};\cFb^\lambda_{s_\lambda-j-1}).
\ee
\end{itemize}
\fin
\et

By specializing \eqref{newthadb} to $j=0$, and denoting by $A_i^\lambda$ the local system monodromy representation along the $i$-th component $\Sigma^\lambda_{s_\lambda,i}$ of the top dimensional stratum of $\Sigma^\lambda$, with transversal Milnor fiber $F^\pitchfork_{s_\lambda,i}$, we get as in Theorem \ref{c14a} the following (see also \cite[Section 3.3]{DS}).
\bc\label{c15a}
\be\label{upbb}
\widetilde{H}^{n-s_\lambda}(F;\bC)_\lambda  \hookrightarrow \bigoplus_i \left( \widetilde{H}^{n-s_\lambda} (F^\pitchfork_{s_\lambda,i};\bC)_\lambda \right)^{A^\lambda_i}.
\ee
Moreover, \eqref{upbb} becomes an isomorphism if $\Sigma^\lambda$ contains no strata of dimension $s_\lambda -1$.
\fin
\ec
\

Taking dimensions, Corollary \ref{c15a} yields (with self-explanatory notations) the following Betti number estimate refining Corollary \ref{c:3}(i):
\be
{b}^\lambda_{n-s_\lambda}(F)\leq \sum_i \dim \widetilde{H}^{n-s_\lambda}(F_{s_\lambda,i}^{\pitchfork};\bC)^{A^\lambda_i} \leq \sum_i {b}^\lambda_{n-s_\lambda}(F^\pitchfork_{s_\lambda,i}).
\ee

\section{Milnor fiber cohomology via iterated slicing}\label{sl} 
In this section, we use nearby and vanishing cycle functors to derive information about the Milnor fiber cohomology via slicing. We also make use here of the {\it vanishing neighborhood of the nearby section} of \cite{Ti-nonisol}, which gives the geometric counterpart of the slicing via perverse sheaves.

\subsection{Setup and notations.}\label{setnot}  \ 
As before, let $n \geq 1$ and consider a nonconstant holomorphic function germ $f:(X,0) \to (\bC,0)$ defined on a pure $(n+1)$-dimensional complex singularity germ $(X,0)$ contained in some ambient $(\bC^N,0)$. We work under the notations and assumptions of \S\ref{nv}, in particular we assume that $\r(X,\bZ)=n+1$ or, equivalently, that $\underline{\bZ}_X[n+1]$ is a strongly perverse sheaf on $X$ in the sense of Definition \ref{sper}(c); see Proposition \ref{p28} and Corollary \ref{cfield}(b). 

Let $l:(\bC^N, 0) \to (\bC, 0)$ be a general linear function germ, i.e., transversal to all the strata of $\cW$ of $X$ except at 0. The restriction of $l$ to $X$, to $f^{-1}(0)$, or to any other subset shall be clear from the context and will also be denoted by $l$. Set
$$Y:=f^{-1}(0), \ H:=l^{-1}(0),  \ f':=f\vert_H, \ Y':=f'^{-1}(0)=Y \cap H.$$
In what follows, we will consider the composition of inclusion maps 
$$\{0\} \overset{i'}{\hookrightarrow} Y'  \overset{u'}{\hookrightarrow} Y \overset{u}{\hookrightarrow} X, 
\ \mbox{ with }\  i :=u' \circ i' : \{0\} \hookrightarrow Y.$$

As in \S \ref{bound}, we consider a Milnor ball $B_\e \subset X$ for $f$ at 0.  The generic choice of $l$ implies that the stratified singular locus $\Sing_{\cW} (l,f)$ of the map germ $(l,f):(X,0) \to (\bC^2,0)$ is the union $\Gamma(l,f) \cup \Sigma$ of the stratified singular locus $\Sigma$ of $f$ with the so-called \textit{polar curve} $\Gamma(l,f)$, see, e.g., \cite{Le2, Ti-nonisol, Ti-book}. It follows that $\Gamma(l,f)$ intersects the fibre $(l,f)^{-1}(0,0)$ only at the origin. In turn, this implies
that the map germ $(l,f)$ is open at the origin of the target, 
and that there exists a fibration outside the discriminant $\Delta := (l,f)(\Gamma(l,f) \cup \Sigma)$, in the following sense: there exist disks (denoted by $D_{*}$) of small enough  radii $0<\gamma' \ll \eta' \ll \epsilon$ such that the map $(l,f) : B \to \bC^2$, with 
$B := B_\e \cap l^{-1}(D_{\eta'}) \cap f^{-1}(D_{\gamma'})$,  
restricts to a locally trivial fibration: 
\be\label{genc}
 (l,f)_{|} : B \setminus (l,f)^{-1}(\Delta) \rightarrow \bC^2 \setminus \Delta .
\ee

It also follows that the fibration \eqref{genc} is also independent of the choices of the constants, cf. \cite{Le2, Ti-nonisol, Ti-book}.
Since the complement of the discriminant $\Delta$ is path connected,  it follows that the fibre of the map \eqref{genc}  is unique.  Let us denote it by $F'$. In particular, $F'$ is the Milnor fibre of $f'$ and also
 $$F' \simeq B_\epsilon \cap f^{-1}(\gamma) \cap l^{-1}(\eta)$$
  for $0<\gamma \ll \eta \ll \epsilon \ll 1$. 

We refer to Remark \ref{rem36} for map germs involving several general functions, and for a different approach by \cite{MP} where the functions are non-generic.

\subsection{Cohomological version of L\^e's attaching result.} \label{Lea} \ 
Here we recall the cohomological version of a well known result of L\^e \cite{Le0} concerning the CW structure of the Milnor fiber. It has already been discussed in this form in \cite{Sa87} or \cite{Ma94}. 

By applying the distinguished triangle \eqref{uno} for the nearby and vanishing cycle functors associated to $l:(Y,0) \to (\bC,0)$ to the complex $\psi_f{\underline{\bZ}_X}$ on $Y$, we get the distinguished triangle
\be\label{dtab}
u'^*\psi_f{\underline{\bZ}_X} \to \psi_{l}\psi_f{\underline{\bZ}_X} \to \varphi_{l}\psi_f{\underline{\bZ}_X} \overset{[1]}{\to}
\ee
Applying the functor $i'^*$ to \eqref{dtab} yields the distinguished triangle
\be\label{dtb}
i^*\psi_f{\underline{\bZ}_X} \to i'^*\psi_{l}\psi_f{\underline{\bZ}_X} \to i'^*\varphi_{l}\psi_f{\underline{\bZ}_X} \overset{[1]}{\to}
\ee
Applying the cohomology functor $H^*(-)$ to \eqref{dtb} and using \eqref{i62} 
yields the following long exact sequence of groups
\be\label{dtc}
\cdots \to H^{k}(F) \to \cH^k(\psi_{l}\psi_f{\underline{\bZ}_X})_0 \to \cH^k(\varphi_{l}\psi_f{\underline{\bZ}_X})_0 \to H^{k+1}(F) \to \cdots
\ee

We include the following well known result and its proof for future reference:
\begin{lemma}
For any integer $k$, 
\be\label{314n}
\cH^k(\psi_{l}\psi_f{\underline{\bZ}_X})_0  \cong  H^k(F') \cong \cH^k(\psi_{f}\psi_l{\underline{\bZ}_X})_0.
\ee 
\end{lemma}
\begin{proof}
With $\gamma, \eta, \epsilon$ as in \S  \ref{setnot}, the generic choice of the linear function $l$ implies that $l^{-1}(\eta)$ is smooth in the ambient space $\bC^{N}$, and transversal to a fixed Whitney stratification $\cW$ of the set germ $(X\m \{0\}, 0)$. Letting $$\widehat{f}:=f\vert_{X \cap l^{-1}(\eta)},$$
there a base change isomorphism (e.g., see \cite[Lemma 4.3.4]{Sc}):
\be\label{transversal}\big(\psi_f{\underline{\bZ}_X}\big)\vert_{f^{-1}(0) \cap l^{-1}(\eta)} \cong \psi_{\widehat{f}}\underline{\bZ}_{X\cap l^{-1}(\eta)}.\ee
Then using \eqref{i62} we get:
\begin{equation*}
\begin{split} \cH^k(\psi_{l}\psi_f{\underline{\bZ}_X})_0  \cong 
H^k(B_\epsilon \cap l^{-1}(\eta) \cap f^{-1}(0); \psi_f{\underline{\bZ}_X} )  & \ \\
 \cong H^k(B_\epsilon \cap l^{-1}(\eta) \cap f^{-1}(0); \psi_{\widehat{f}}\underline{\bZ}_{X\cap l^{-1}(\eta)} ) 
 \cong H^k(B_\epsilon \cap l^{-1}(\eta) \cap f^{-1}(\gamma); \bZ) 
 &  \cong H^k(F').
\end{split}
\end{equation*}
To show the isomorphism $\cH^k(\psi_{f}\psi_l{\underline{\bZ}_X})_0 \cong  H^k(F')$ we repeat the above procedure. What makes it possible is the genericity of $l$ as discussed in \S \ref{genc}. This genericity implies, roughly speaking, that commuting the functors $\psi_{f}$ and $\psi_l$ yields the isomorphism \eqref{314n}.
\end{proof}

By combining \eqref{dtc} and \eqref{314n}, we get that 
\be\label{315n}
\cH^k(\varphi_{l}\psi_f{\underline{\bZ}_X})_0 \cong H^{k+1}(F, F').
\ee
Since $l$ is generic, the origin $\{0\}$ is an isolated point in the support of the strongly perverse sheaf ${^p\varphi}_{l}{^p\psi}_f(\underline{\bZ}_X[n+1])$ on $Y'$, and hence the cohomology of $i'^*{^p\varphi}_{l}{^p\psi}_f(\underline{\bZ}_X[n+1])$ is concentrated in degree zero. Moreover, since 
\begin{equation*}
\begin{split}
\cH^k(\varphi_{l}\psi_f{\underline{\bZ}_X})_0 \cong H^{k}( i'^*\varphi_{l}\psi_f{\underline{\bZ}_X} ) &\cong H^{k-n+1} \big(  i'^*{^p\varphi}_{l}{^p\psi}_f(\underline{\bZ}_X[n+1])\big) \\ &\cong H^{k-n+1} \big(  (i')^!{^p\varphi}_{l}{^p\psi}_f(\underline{\bZ}_X[n+1])\big),
\end{split}
\end{equation*}
we get that $\cH^k(\varphi_{l}\psi_f{\underline{\bZ}_X})_0=0$ for all $k \neq n-1$, and 
$$\cH^{n-1}( \varphi_{l}\psi_f{\underline{\bZ}_X} )_0 \cong H^n(F,F') \cong \bZ^{\tau_f}$$ is free.

The long exact sequence \eqref{dtc} together with \eqref{314n} and \eqref{315n} then yield the isomorphisms
\be\label{eq:slicing1} H^k(F) \cong H^k(F'), \ \ \ \text{ for } \ k<n-1,\ee
and an exact sequence
\be\label{eq:slicing2} 0 \to H^{n-1}(F) \to H^{n-1}(F') \to \bZ^{\tau_f} \to H^n(F) \to 0,\ee
where
$$\tau_f={\rm int}_0(\Gamma(l,f), f^{-1}(0))$$ 
is a polar intersection multiplicity at $0$  (see \S \ref{setnot} for the definition of the polar locus $\Gamma(l,f)$).
Note that by \eqref{eq:slicing2} one has that $b_n(F) \leq \tau_f$, where $b_n(F)$ is the $n$-th Betti number  of the Milnor fiber of $f$ at the origin. We will give a sharper bound in \eqref{c23b} below.


\subsection{Milnor fiber cohomology and the vanishing neighborhood of the nearby section} \label{Tib} \ 
We give here a sheaf-theoretic version of the cohomological results obtained by Tib\u{a}r in \cite[Theorem 2.2]{Ti-nonisol}, which will play a fundamental role in this paper, and which we could not locate in this form in the literature.

We start by evaluating the distinguished triangle \eqref{uno} for the nearby and vanishing cycle functors associated to $l:(Y,0) \to (\bC,0)$ on the complex $\varphi_f{\underline{\bZ}_X}$ on $Y$, to get the distinguished triangle on $Y'$
\be\label{dta2}
u'^*\varphi_f{\underline{\bZ}_X} \to \psi_{l}\varphi_f{\underline{\bZ}_X} \to \varphi_{l}\varphi_f{\underline{\bZ}_X} \overset{[1]}{\to}
\ee
Applying $i'^*$, and then taking the long exact sequence obtained by applying the cohomology functor $H^*(-)$, one then gets by using \eqref{i62} the long exact sequence
\be\label{lesT}
\cdots \to \widetilde{H}^k(F) \to \cH^k(\psi_{l}\varphi_f{\underline{\bZ}_X})_0
\to \cH^k(\varphi_{l}\varphi_f{\underline{\bZ}_X})_0  \to \widetilde{H}^{k+1}(F) \to \cdots
\ee
Since ${^p\psi}_{l}{^p\varphi}_f(\underline{\bZ}_X[n+1])$ is a perverse sheaf on $Y'$, the support condition for perverse sheaves yields that
$$\cH^k(\psi_{l}\varphi_f{\underline{\bZ}_X})_0 = H^{k-n+1}\big(i'^*{^p\psi}_{l}{^p\varphi}_f(\underline{\bZ}_X[n+1])\big) =0 \ \ \ \text{if} \ k>n-1.$$
Furthermore, since $l$ is generic, we obtain as before that $\cH^k(\varphi_{l}\varphi_f{\underline{\bZ}_X})_0=0$ for all $k \neq n-1$. 
Altogether, we get from \eqref{lesT} isomorphisms
\be\label{eq:slicing3} \widetilde{H}^k(F) \cong \cH^k(\psi_{l}\varphi_f{\underline{\bZ}_X})_0 \ \ \ \text{if} \ k<n-1,\ee
and an exact sequence
\be\label{eq:slicing4} 0 \to \widetilde{H}^{n-1}(F) \to \cH^{n-1}(\psi_{l}\varphi_f{\underline{\bZ}_X})_0
\to \cH^{n-1}(\varphi_{l}\varphi_f{\underline{\bZ}_X})_0  \to \widetilde{H}^{n}(F)  \to 0.\ee

Let us next give a geometric interpretation of \eqref{eq:slicing3} and  \eqref{eq:slicing4}. Since $\psi_{l}$ is a functor of triangulated categories, by applying it to the distinguished triangle on $Y$ 
\be\label{bas}\underline{\bZ}_Y \to \psi_{f} \underline{\bZ}_X \to \varphi_f \underline{\bZ}_X \overset{[1]}{\to}\ee
we get a new distinguished triangle
\be
\psi_{l}\underline{\bZ}_Y \to \psi_{l}\psi_{f} \underline{\bZ}_X \to \psi_{l}\varphi_f \underline{\bZ}_X \overset{[1]}{\to}
\ee
of constructible complexes on $Y'$. Applying the functor $\cH^*(-)_0$ of taking the stalk cohomology at the origin, one gets  the long exact sequence
\be\label{lesT2}
\cdots \to \cH^k(\psi_{l}\underline{\bZ}_Y)_0 \to  \cH^k(\psi_{l}\psi_{f} \underline{\bZ}_X)_0 \to \cH^k(\psi_{l}\varphi_f \underline{\bZ}_X)_0 \to \cdots
\ee
We have that 
\be\label{215b}\cH^k(\psi_{l}\underline{\bZ}_Y)_0 \cong H^k\big(B_\epsilon \cap Y \cap  l^{-1}(\eta)\big) \cong H^k\big(B_\epsilon \cap f^{-1}(D_\gamma) \cap  l^{-1}(\eta)\big),\ee
where $D_\gamma \subset \bC$ is a closed disk of radius $\gamma$ centered at the origin.
Following the notations from \cite{Ti-nonisol} and from \S \ref{setnot}, we set
$$F'_D:=B_\epsilon \cap f^{-1}(D_\gamma) \cap  l^{-1}(\eta)$$ and note that $F \cap F'_D=F'$. In the last isomorphism of \eqref{215b} we used the fact that $F'_D$ deformation retracts to the complex link $\lk(Y,0)=B_\epsilon \cap Y \cap  l^{-1}(\eta)$ of the hypersurface  $Y=f^{-1}(0)$ at the origin. 
Plugging \eqref{314n} and \eqref{215b} in \eqref{lesT2}, we obtain for any $k \leq n-1$ an isomorphism
\be\label{geoi}
\cH^k(\psi_{l}\varphi_f{\underline{\bZ}_X})_0 \cong H^{k+1}(F'_D,F').
\ee
We can therefore restate \eqref{eq:slicing3} as an isomorphism
\be\label{eq:slicing3b} 
\widetilde{H}^k(F) \cong H^{k+1}(F'_D,F') \ \ \ \text{if} \ k<n-1.
\ee
Furthermore, since by excision we get $H^*(F\cup F'_D,F) \cong H^*(F'_D,F')$,
the long exact sequence \eqref{eq:slicing4} can be identified with the long exact sequence of (reduced) cohomology of the pair $(F\cup F'_D,F)$. So in particular we have the identification 
$$\cH^{n-1}(\varphi_{l}\varphi_f{\underline{\bZ}_X})_0  \cong \widetilde{H}^n(F\cup F'_D),$$
and  \eqref{eq:slicing4} becomes the exact sequence from \cite[Theorem 2.2]{Ti-nonisol}, namely
\be\label{eq:slicing4b} 0 \to \widetilde{H}^{n-1}(F) \to H^{n}(F'_D,F')
\to \widetilde{H}^n(F\cup F'_D)  \to \widetilde{H}^{n}(F)  \to 0.\ee
For future reference, let us record here the fact that the above discussion also shows that the reduced cohomology $\widetilde{H}^*(F\cup F'_D)$ is concentrated in degree $n$ (where it is free).

\medskip

Let  $T$ be a small tubular neighborhood of the complex link $\lk(\Sigma,0)=B_{\e}\cap \Sigma \cap  l^{-1}(\eta)$ of $\Sigma$ in the slice $l^{-1}(\eta)$. 
By retraction we get (as in \cite{Ti-nonisol}) the isomorphism:
\begin{equation}\label{eq:retr}
  H^* (F'_D, F')  \cong H^* (T, T\cap F'), 
\end{equation}
 which provides us with a replacement of the pair $(F'_D, F')$ appearing in \eqref{eq:slicing3b} and \eqref{eq:slicing4b} by the pair  $(T, T\cap F')$.

\begin{definition}\label{d:vanish}
 We call the pair $(T, T\cap F')$ the {\it vanishing neighborhood of the nearby section}, and  we call $H^* (T, T\cap F')$ the {\it vanishing cohomology of the nearby section} of $f$.
\end{definition}

The Milnor monodromy $h$ of the cohomology $H^*(F)$ is induced by a geometric monodromy which acts on $F$ by moving along a small circle $y(t) =\exp(2i\pi t)\gamma$, and on the slice $F' =B_{\e} \cap f^{-1}(\gamma)\cap l^{-1}(\eta)$ by moving along the circle $\{ y(t) =\exp(2i\pi t)\gamma\} \cap \{ x=\eta\} \subset \bC^{2}$. It acts as the identity on  $F'_D$  and on the neighborhood $T$.

There is another geometric monodromy, defined by moving the slice $l=\eta$ around the circle $\exp(2i\pi t)\eta$ for $t\in [0,1]$.
We call it  {\it the sectional $l$-monodromy}. This acts on $F'$, on  $F'_D$,  and on $T\cap F'$. It acts as the identity on $F$.
Its induced action on the $\bZ$-cohomology groups of all these spaces will be denoted by $L$, which in the previous notations can also be identified with the action on $\cH^*(\psi_{l}\varphi_f{\underline{\bZ}_X})_0$ induced from the monodromy action on $\psi_{l}$.

In particular the sectional monodromy acts on the exact sequence  \eqref{eq:slicing4b}. Looking only to the left hand side of it, by the identification \eqref{eq:retr},we get the following monomorphism, cf. \cite[Corollary 2.4]{Ti-nonisol}:
\begin{equation}\label{eq:monomorphism0}
\widetilde{H}^{n-1}(F) \hookrightarrow  \ker \big[L -\id \mid  H^{n}(T, T\cap F')\big]
\end{equation}


\subsection{Iterated slicing, iterated nearby and vanishing cycles} \label{itsl} \     
We focus now on the computation of the vanishing cohomology of the nearby section which, as can be seen from \eqref{eq:slicing3b}, yields the  cohomology of the Milnor fiber $F$, except in the two top dimensions. 

In what follows, we calculate the Milnor fiber cohomology using repeated slicing by general hyperplanes $H_{k}:= \{l_{k} =0\}$ for $1\le k\le s-2$, with each $l_k$ general linear functions with $l_k(0)=0$, as in \S \ref{Lea} and \S \ref{Tib}.

Let $X^{(k)}:=X\cap H_{1} \cap \cdots \cap H_{k}$, with $X^{(0)}=X$, and 
 consider the function germ $$f^{(k)} : (X^{(k)}, 0) \to (\bC,0)$$ with  Milnor fiber $F^{(k)}$ at the origin. In particular, $F^{(1)}=F'$ and we set $F^{(0)}:=F$. Let $Y^{(k)}:=(f^{(k)})^{-1}(0)=Y \cap H_1 \cap \cdots \cap H_k$.
By the genericity of the hyperplane slices, $\Sigma^{(k)} :=\Sigma \cap H_{1} \cap \cdots \cap H_{k}$ is the singular locus of $f^{(k)}$ and its Whitney stratification is induced by that of $\Sigma$.

The Milnor fiber $F^{(k)}$ is also identified with $B_{\e}\cap (l_{k}, f^{(k-1)})^{-1}(\eta, \gamma)$, the tube $F^{(k)}_D$ is defined  similarly to  $F'_{D}$ in \S \ref{Tib}, 
and $T^{(k)}$ is a tubular neighborhood of the complex link $\lk(\Sigma^{(k)},0)$ of the singular locus $\Sigma^{(k)}$, where $T^{(0)} := T$. As in \S \ref{Tib}, the reduced integral cohomology $\widetilde H^{*}(F^{(k-1)}\cup F^{(k)}_D)$  is concentrated in degree $n-k+1$.  

The case of a $1$-dimensional singular locus having been considered before, 
we focus now on the higher dimensional case $n>s\ge 2$. 
Combining the slice isomorphism \eqref{eq:slicing1} on the one hand, with the isomorphism \eqref{eq:slicing3b} and the exact sequences \eqref{eq:slicing4b}
on the other hand, we get the following:
\begin{proposition}\label{t:izos}
Let $n>s\ge 2$.  
 For each fixed  $2\le k\le s$, there are isomorphisms
\[  
  \widetilde H^{n-k}(F)  \cong    \widetilde H^{n-k}(F^{(j)})  \cong H^{n-k+1}(T^{(j)}, T^{(j)}\cap F^{(j+1)})
\]
for any $j= 1, \ldots, k-2$. Moreover,  for $1\le k\le s-1$ we have the following 
exact sequence  
 \begin{multline*} \ \ \ 0\to \widetilde H^{n-k} (F^{(k-1)}) \to H^{n-k+1}(T^{(k-1)}, T^{(k-1)}\cap F^{(k)})  \\ \to \widetilde H^{n-k+1}(F^{(k-1)}\cup F^{(k)}_D)
 \to \widetilde H^{n-k+1}(F^{(k-1)})  \to  0, \ \ \ \ \end{multline*}
 where $\widetilde H^{n-k} (F)   \cong    \widetilde H^{n-k}(F^{(k-1)})$.
\fin
\end{proposition}

The groups $H^{q}(T^{(j)}, T^{(j)}\cap F^{(j+1)})$ appearing in Proposition \ref{t:izos} can be described by an iteration of nearby and vanishing cycle functors, as follows. 
\begin{proposition}\label{t:it}
For any integers $q\geq 1$ and $j \geq 1$, 
\be\label{MP1}
H^q(F^{(j)}) \cong \cH^q(\psi_{l_j}\ldots \psi_{l_1}\psi_f \underline{\bZ}_X)_0,
\ee
\be\label{MP2}
H^{q}(T^{(j)}, T^{(j)}\cap F^{(j+1)})\cong \cH^q(\psi_{l_{j+1}}\psi_{l_j}\ldots \psi_{l_1}\varphi_f \underline{\bZ}_X)_0.
\ee
\end{proposition}
\begin{proof}
The isomorphism \eqref{MP1} is the extension of \eqref{314n}  by iteration, and follows from it by induction using the genericity of the linear functions $l_1,\ldots, l_j$ (see also Remark \ref{rem36}).

We now prove \eqref {MP2}. By \eqref{eq:retr}, we have an isomorphism
\be 
H^{q}(T^{(j)}, T^{(j)}\cap F^{(j+1)}) \cong H^q(F_D^{(j+1)}, F^{(j+1)}),
\ee
with 
\be\label{jpl1} F_D^{(j+1)}:=B_\epsilon \cap (f^{(j)})^{-1}(D_\gamma) \cap l_{j+1}^{-1}(\eta_{j+1})\simeq B_\epsilon \cap (f^{(j)})^{-1}(0) \cap l_{j+1}^{-1}(\eta_{j+1})=\lk(Y^{(j)},0),
\ee
with $\gamma, \eta_{j+1}$ sufficiently small.
To prove \eqref{MP2} it is therefore sufficient to show that  
\be\label{MP3}
H^q(F_D^{(j+1)}, F^{(j+1)})\cong \cH^q(\psi_{l_{j+1}}\psi_{l_j}\ldots \psi_{l_1}\varphi_f \underline{\bZ}_X)_0.
\ee 
Applying the functor $\psi_{l_{j+1}}\psi_{l_j}\ldots \psi_{l_1}$ to the distinguished triangle \eqref{bas} on $Y$, we get a new distinguished triangle on $Y^{(j+1)}$:
\be\label{bas2}\psi_{l_{j+1}}\ldots \psi_{l_1}\underline{\bZ}_Y \to \psi_{l_{j+1}}\ldots \psi_{l_1}\psi_{f} \underline{\bZ}_X \to \psi_{l_{j+1}}\ldots \psi_{l_1}\varphi_f \underline{\bZ}_X \overset{[1]}{\to}\ee
By applying the cohomological functor $\cH^*(-)_0$ to \eqref{bas2}, we get a long exact sequence
\be\label{lesT25}
\cdots \to \cH^q(\psi_{l_{j+1}}\ldots \psi_{l_1}\underline{\bZ}_Y)_0 \to  \cH^q(\psi_{l_{j+1}}\ldots \psi_{l_1}\psi_{f} \underline{\bZ}_X)_0 \to \cH^q(\psi_{l_{j+1}}\ldots \psi_{l_1}\varphi_f \underline{\bZ}_X)_0 \to \cdots
\ee
As already shown in \eqref{MP1}, 
\be\label{2155a}  \cH^q(\psi_{l_{j+1}}\ldots \psi_{l_1}\psi_{f} \underline{\bZ}_X)_0  = H^q(F^{(j+1)}),\ee
and the same formula yields that
\be\label{2155b}
\cH^q(\psi_{l_{j+1}}\ldots \psi_{l_1}\underline{\bZ}_Y)_0 \cong H^q\big(B_\epsilon \cap Y \cap  l_1^{-1}(\eta_1) \cap \cdots l_{j+1}^{-1}(\eta_{j+1})\big)
\ee
for small enough $\eta_1, \ldots, \eta_{j+1}$.
By the genericity of the linear functions $l_1, \ldots, l_{j+1}$, there is a homotopy equivalence $$B_\epsilon \cap Y \cap  l_1^{-1}(\eta_1) \cap \cdots \cap l_{j+1}^{-1}(\eta_{j+1}) \simeq 
B_\epsilon \cap Y \cap  l_1^{-1}(0) \cap \cdots \cap l_{j}^{-1}(0) \cap l_{j+1}^{-1}(\eta_{j+1}) = \lk(Y^{(j)},0),$$
defined by moving $(\eta_1, \ldots, \eta_j, \eta_{j+1})$ along a straight path to $(0, \ldots, 0, \eta_{j+1})$.
Combining this fact with \eqref{jpl1} and \eqref{2155b}, we therefore have an isomorphism
\be\label{2156a}
\cH^q(\psi_{l_{j+1}}\ldots \psi_{l_1}\underline{\bZ}_Y)_0 \cong H^q(F_D^{(j+1)}).
\ee
Plugging \eqref{2155a} and \eqref{2156a} into the long exact sequence \eqref{lesT25} yields \eqref{MP3}, thus completing the proof.
\end{proof}

\begin{remark}\label{rem36}
Note that, by the genericity of the linear forms $l_{i}$,  the map $(f, l_{1}, \ldots , l_{j})$, for $j\ge 1$, defines a stratified  isolated singularity at the origin (e.g., it is an ICIS in case of $X = \bC^{n+1}$), and therefore it is an open map at the origin of the target and defines a stratified fibration outside its discriminant (which is actually included in a hypersurface).  This is a particular case of  a map ``sans \' eclatement''  in the terminology of Sabbah \cite{Sa}. Consequently, $\cH^q(\psi_{l_j}\ldots \psi_{l_1}\psi_f \underline{\bZ}_X)_0$ in formula \eqref{MP1} is independent of the order of the nearby cycle functors in the sequence. A non-generic situation occurs in \cite[Section 3]{MP}, where the iterated nearby cycles do not commute anymore, where the map may not be open\footnote{We may refer to \cite{JT} for a study of images of map germs in relation with singular fibrations.} anymore, and where the stratified fibration has a very special meaning depending on the {\it order} of the functions in the map. 
A formula similar to \eqref{MP1} also holds in this general situation after fixing an ordering of the defining functions, see \cite[Formula (8)]{MP}.   
\end{remark}

\subsection{Betti bounds and polar multiplicities}\label{ss:bettibounds1} \ 
In this section we indicate how the slicing technique yields Betti bounds for the Milnor fiber. 

We begin with describing an upper bound of the top Betti number of the Milnor fiber. While this result is known (e.g., see \cite[Corollary 2.3]{Ti-nonisol}), for completeness we include here a proof in the spirit of \S  \ref{Tib}. 
\bc\label{cor32}\rm (\cite[Corollary 2.3]{Ti-nonisol},  \cite[Theorem 3.3]{Ma95}) \it
\be\label{c23b}
{b}_n(F) \leq  \lambda^0+ {b}_{n}(\lk(X,0)),\ee
where  $\lambda^0:=\tau_f - \tau_l={\rm int}_0(\Gamma, f^{-1}(0))-{\rm int}_0(\Gamma, l^{-1}(0)).$
\ec
\begin{proof}
By applying the functor  $\varphi_{l}$  to the distinguished triangle \eqref{bas} on $Y$,
we get a new distinguished triangle
\[
\varphi_{l}\underline{\bZ}_Y \to \varphi_{l}\psi_{f} \underline{\bZ}_X \to \varphi_{l}\varphi_f \underline{\bZ}_X \overset{[1]}{\to}
\]
on $Y'$, which, upon applying the functor $\cH^*(-)_0$ of taking the stalk cohomology at the origin and the fact mentioned earlier that $\cH^k(\varphi_{l}\varphi_f{\underline{\bZ}_X})_0=0$ for all $k \neq n-1$, yields the short exact sequence
\be\label{lesT3}
0 \to \cH^{n-1}(\varphi_{l}\underline{\bZ}_Y)_0 \to  \cH^{n-1}(\varphi_{l}\psi_{f} \underline{\bZ}_X)_0 \to \cH^{n-1}(\varphi_{l}\varphi_f \underline{\bZ}_X)_0 \to 0.
\ee
Using the notation from \eqref{eq:slicing2}, we get from \eqref{eq:slicing4b} and \eqref{lesT3} together with \eqref{i62} that 
\be\label{c23}
\widetilde{b}_n(F) \leq \tau_f - \widetilde{b}_{n-1}(\lk(Y,0)).
\ee
On the other hand, since $\underline{\bZ}_X[n+1]$ is strongly perverse on $X$, the complex link $\lk(X,0)$ of $X$ at the origin has its reduced cohomology concentrated in degree $n$ (e.g., see \cite[Corollary 10.6.3]{Max}), where it is free (by the freeness of lowest degree costalks). Similarly, since $\underline{\bZ}_Y[n]$ is strongly perverse on $Y$ (cf. \S\ref{nv}), the reduced cohomology of the complex link $\lk(Y,0)$ of $Y$ at the origin is concentrated in degree $n-1$, where it is free (again, by the freeness of costalks). The long exact sequence for the reduced cohomology of the pair $\big(\lk(X,0), \lk(Y,0)\big)$ then yields that $$H^n\big(\lk(X,0), \lk(Y,0)\big) \cong \bZ^{\tau_l}$$ is free, with $$\tau_l= \widetilde{b}_{n}(\lk(X,0))+ \widetilde{b}_{n-1}(\lk(Y,0)).$$
It is known by work of L\^e (see, e.g., \cite[Facts 1.1(b,c)]{Ti-nonisol} and the references therein) that $$\tau_l={\rm int}_0(\Gamma, l^{-1}(0))$$ 
is a corresponding polar intersection number for $l$. Our claim follows now from  \eqref{c23}.
\end{proof}

 The following two relations are consequences of iterated slicing and repeated application of  \eqref{eq:slicing1} and \eqref{eq:slicing2}:
 \begin{equation}
  b_{n-k}(F) =  b_{n-k}(F^{(k-1)})  \mbox{ \ \ and \ \  }  b_{n-k}(F) \le b_{n-k}(F^{(k)}).
\end{equation}
These yield  the inequality: 
\begin{equation}\label{eq:lef_ineq}
b_{n-k}(F^{(k-1)}) \le b_{n-k}(F^{(k)}), \ \  \mbox{ \ \ for \ \  } k\ge 1,
\end{equation}
which can be used to get the following generalization of Corollary \ref{cor32} (compare also with \cite[Corollary 2.3 ]{Ti-nonisol} and \cite{Ma0}): 
\begin{corollary}\label{t:bettibounds}
Let $s\ge 1$. For any $k = 0,\ldots ,s-1$, we have the bounds:
\begin{equation}\label{eq:lefk}
 b_{n-k}(F) \le b_{n-k}(F^{(k)}) \le   \lambda^{k} + b_{n-k}(\lk^{k}(X,0)), 
\end{equation}
  where 
\begin{equation}\label{eq:lambda}
\lambda^{k} :=
  \int_0(\Gamma(l_{k},f^{(k)}) , (f^{(k)})^{-1}(0)) - \int_0(\Gamma(l_{k},f^{(k)}), (l_{k})^{-1}(0)), 
  \end{equation}
  \fin
\end{corollary}

\begin{remark}
 The numbers $\lambda^{*}$ which occur in \eqref{eq:lambda} are the analogues of the sectional Milnor numbers $\mu^{*}$ defined by Teissier \cite{Te, Te-p}, in the sense that $\mu^{*}$ verify the same equality  \eqref{eq:lambda} in case $X =\bC^{n+1}$ and $f$ has an {\it isolated singularity}.  In Massey's terminology, $\lambda^{*}$ are called ``L\^{e} numbers'', see e.g.  \cite{Ma94}.
 
 \end{remark}


   \section{The geometric computation of  $\widetilde H^{n-s}(F)$}\label{s:Hn-s} 

 In this section we compute the lowest (possibly nontrivial) group $\widetilde H^{n-s}(F)$  by using the isomorphism with the vanishing cohomology of the nearby section.\footnote{All these computations also apply for an admissible deformation of $f$.}  
 Throughout this section we continue to use the previous assumptions, namely that $(X,0)$ is a pure $n+1$ dimensional space germ with $\r(X,\bZ)=n+1$,  $n\ge 3$ and $\dim \Sing_{\cW} f = s\ge 2$.

From Proposition \ref{t:izos} for $k=s$ and $j=s-2$, we have:
 
\begin{equation}\label{eq:new}
 \widetilde H^{n-s}(F)  \cong \widetilde H^{n-s}(F^{(s-2)})  \cong  H^{n-s+1}(T^{(s-2)}, T^{(s-2)}\cap F^{(s-1)}).
\end{equation}

 Consequently, 
 $$\widetilde H^{n-s}(F) \cong \widetilde H^{m-2}(F_{g})$$ 
  where $g:= f_{|\cH}$, $\cH := H_{1}\cap \cdots \cap H_{s-2}$, $m:= n-s+2$,  with $\dim \Sing_{\cW\cap \cH} (g) =2$, and $F_{g}$ denotes the Milnor fibre of the function germ $g$.

By the above isomorphism, it is sufficient  to compute $\widetilde H^{m-2}(F_{g})$ in order to obtain $\widetilde H^{n-s}(F)$, and therefore the remainder of this section will be devoted to this computation. For the sake of simplicity we 
will use the notation  $f$ instead of $g$,  assume that $\dim_{0}\Sing_{\cW}= 2$, and thus compute $\widetilde H^{n-2}(F)$.


 \subsection{Computing the vanishing cohomology of the nearby section}\label{s:s=2}  

Let $n\ge 3$ and $s =\dim \Sing f =2$. As before, we  denote by $\Sigma_{2}$ the union of the strata of $\Sing f$ of dimension 2, and by $\Sigma_{1}$  the union of the 1-dimensional strata of  $\Sing f$. We may have singular 1-dimensional strata which are inside or outside $\overline{\Sigma_{2}}$.
   
The cohomology of the Milnor fiber $\widetilde H^{*}(F)$ is known to be concentrated only in degrees $n, n-1$ and $n-2$.
From \eqref{eq:slicing3b} and \eqref{eq:slicing4b} we know that it is computed by the vanishing cohomology of the nearby section in dimension $n-2$, namely    we have  the isomorphism: 
\[\widetilde H^{n-2}(F)  \cong H^{n-1}(T, T\cap F^{(1)}),\] 
and the inclusion:
\[\widetilde H^{n-1}(F)  \hookrightarrow H^{n}(T, T\cap F^{(1)}),\] 
where we recall that $T$ is a small tubular neighborhood of $B_{\e}\cap \Sing f \cap  l^{-1}(\eta)$ in the slice $l^{-1}(\eta)$,
and $F^{(1)}:= B_{\e}\cap l^{-1}(\eta) \cap f^{-1}(\gamma)$. 

We therefore focus on the computation of the vanishing cohomology $H^*(T, T\cap F^{(1)})$ of the nearby section.

 The singular slice  $$S := \overline{\Sigma_{2}} \cap l^{-1}(\eta)$$ is a union of curves, so let $S = S_1 \cup \ldots \cup S_{\rho}$ be its decomposition into irreducible components.  The finite set of points
$\Sigma_{1} \cap l^{-1}(\eta)$ is the union $R\sqcup \cup_{i=1}^{\rho} Q_{i}$, where $R$  is the set of all {\it isolated singularities} outside  $S$  of the slice  $B_{\e}\cap \Sing f \cap  l^{-1}(\eta)$, 
and  where $Q_{i}:= S_i \cap \Sigma_{1}$ will be called the set of {\it special points} of  $S_i$.
Note that if the curve components $S_i$ and $S_j$ intersect, for $i\not= j$, then the intersection points belong to $\Sigma_{1}$, and thus $S_i\cap S_{j}  \subset Q_{i}$ for any $i\not= j$.

Let   $S^* := S \m  \cup_{i=1}^{\rho} Q_{i}$,  and $S_{i}^* := S_{i} \m  Q_{i}$ for each $i$. 
 
Let $B_{r}$ be a Milnor ball at  the point $r\in R$ in the slice $l^{-1}(\eta)$, and let $T_S$ denote a small tubular neighborhood of $S$ in the same slice. Then $F_{r}:= B_{r}\cap F^{(1)}$ is the Milnor fibre of an isolated hypersurface singularity at $r$, and therefore its reduced cohomology is concentrated in the top dimension.
With these notations one has:

\begin{lemma}
  \begin{equation}\label{eq:decomp}
  \begin{split}
\widetilde H^{n-2}(F) \cong  H^{n-1}(T, T\cap F^{(1)})  \cong  H^{n-1}(T_S, T_S \cap F^{(1)}) \hspace{2cm}   \\
\widetilde  H^{n-1}(F) \hookrightarrow   H^{n}(T, T\cap F^{(1)})  \cong  H^{n}(T_S, T_S \cap F^{(1)}) \oplus  \bigoplus_{r\in R} H^{n-1}(F_{r})
 \end{split}
\end{equation}
\end{lemma}
\begin{proof}
   We have the direct sum decomposition: 
 \[ H^{*}(T, T\cap F^{(1)})  \cong \bigoplus_{r\in R} H^{*}(B_{r}, B_{r}\cap F^{(1)}) \oplus H^{*}(T_S, T_S \cap F^{(1)}).
\]
 Since $H^{*}(B_{r}, B_{r}\cap F^{(1)}) \cong H^{*-1}(F_{r})$  is concentrated in dimension $* =n$, the first term is a direct summand of $H^{n}(T, T\cap F^{(1)})$ and the assertion follows. 
\end{proof}


  \subsection{Cohomology of  $(T_S,T_S \cap F^{(1)})$} \label{ss:cw}\ 
 To study the cohomology of the pair $(T_S,T_S \cap F^{(1)})$, we inspire ourselves from the technique displayed  in \cite{ST-deform} and \cite{ST-vanishing}.
 
  
The small tubular neighborhood $T_S$ is the union $T_S^{*}\cup B_{Q}$ of  
the tubular neighborhood $T_S^{*}$ of  $S^*$,  together with $B_Q := \bigsqcup_{q\in Q} B_q$, where $B_q$ denotes a  small enough Milnor ball at $q\in Q$ in the slice $l^{-1}(\eta)$.

  Each component $S_i$ has a generic transversal Milnor fiber $F_{i}^{\pitchfork}$, on the cohomology of which
 $\pi_{1}(S^{*}_{i})$ acts.  Let  $\tmu_i$ denote its Milnor number.
 
The pair 
$$(T_S,T_S \cap F^{(1)}) = (T_S^{*}\cup B_Q, (T_S^{*}\cap F^{(1)}) \cup (B_Q\cap F^{(1)}))$$
  comes into a relative Mayer-Vietoris long exact sequence 
\begin{equation}\label{eq:mv}
\begin{split}
 \cdots \to H^{*}(T_S,T_S \cap F^{(1)})  \to H^{*}(T_S^{*}, T_S^{*}\cap F^{(1)})  \oplus  H^{*}(B_Q, B_Q\cap F^{(1)})   \to  \\
  \to     H^{*}(T_S^{*}\cap B_Q, (T_S^{*}\cap F^{(1)}) \cap (B_Q\cap F^{(1)}))     \to \cdots \ \ \ \ \ \ \ \ \ \ \ \ \ \ \ \ \ \ \ \ \  
 \end{split}
\end{equation}
that we analyze in the following. 

\subsubsection{The intersection term  $(T_S^{*}\cap B_Q, (T_S^{*}\cap F^{(1)}) \cap (B_Q\cap F^{(1)}))$}\label{ss:index}

 This is a disjoint union  of pairs localised at points  $q \in Q$, namely  one pair for each local irreducible branch of the germ $(S,q)$. 
 
 Let  $K_q$ be the set of indices for the irreducible branches at $q\in Q$. 
We use the same set of indices $K_{q}$ for the corresponding small loops around $q$ in $S^{*}$.

When we are counting the local irreducible branches at some point $q\in Q_i$ on a specified component $S_i$ then the set of indices for the local branches of the curve germs $(S_i, q)$ will be denoted by $K_{qi}$. 
 
With these notations we get the following decomposition:
\begin{equation}\label{eq:sumdecomp}
 H^{*}(T_{S}^{*}\cap B_Q, (T_{S}^{*}\cap F^{(1)}) \cap (B_Q\cap F^{(1)})) \cong  \bigoplus_{q\in Q} 
\bigoplus_{k\in K_q} H^{*}(\cZ_k, \cC_k).
\end{equation}

One such local pair $(\cZ_k, \cC_k)$ is the bundle over the  component of the link of the corresponding irreducible branch of the curve germ $(S, q)$, having as fiber the local transversal Milnor data $(\tE_k, \tF_k)$. These data depend only on the component $S_i$ containing the local branch indexed by $k$,  and therefore  we  have $(\tE_k, \tF_k) = (\tE_i, \tF_i)$, and in particular the Milnor numbers equality $\tmu_{k} = \tmu_{i}$.

\medskip
\label{eq:circbundle}

The relative cohomology groups in the  direct sum decomposition \eqref{eq:sumdecomp} are concentrated in dimensions $n-1$ and $n$, and depend on the local system  monodromy $\nu_k$ along the link $\cC_k$, by the following \emph{Wang sequence}\index{Wang sequence}:
\begin{equation}\label{eq:wang2}
 0 \rightarrow H^{n-1} (\cZ_k,\cC_k) \rightarrow H^{n-1} (\tE_k,\tF_k)
\stackrel{\nu_k -\id}{\rightarrow} H^{n-1} (\tE_k,\tF_k) \rightarrow H^{n}
(\cZ_k,\cC_k) \rightarrow 0,
\end{equation}
from which one gets the isomorphisms: 
\[
H^{n-1}(\cZ_k,\cC_k) \cong {\ker} \; (\nu_k - \id \mid  H^{n-2}(\tF_k)), \ \ \ \   H^{n}(\cZ_k,\cC_k ) \cong {\coker} \; (\nu_k - \id \mid  H^{n-2}(\tF_k))
\]
  for any $k\in K_{q}$, $q\in Q$.

\subsubsection{Cohomology of $(B_Q, B_Q\cap F^{(1)})$}\label{ss:term}
Since $B_Q$ is a disjoint union, one has the direct sum decomposition $$H^{*}(B_Q, B_Q\cap F^{(1)}) \cong \bigoplus_{q\in Q} H^*(B_q,F_q),$$ into the local Milnor data of the hypersurface germs $(f^{-1}(0)\cap l^{-1}(\eta), q)$ which have 1-dimensional singular locus. Therefore the relative cohomology $H^k(B_q,F_q)\cong \widetilde H^{k-1}(F_q)$ is non-zero only for $k=n-1$ and $k=n$.

\subsubsection{Cohomology of  $(T_{S}^{*}, T^{*}_{S}\cap F^{(1)})$} \label{ss:cw2} 

Each irreducible  1-dimensional component $\Sigma^{0}_{i}$ of $\Sing f_{|l=0}$ is a curve germ, and hence its link $\partial B_{\e} \cap \Sigma^{0}_{i}$ is a circle. Our $S \sqcup R$ is an {\it admissible deformation} of $\Sing f_{|l=0}$ in the sense\footnote{Moreover,  $f_{|l=\eta}$ is an admissible deformation of $f_{|l=0}$ in $B_{\e}$.} of \cite[\S 3.1]{ST-deform}, i.e., an isotopy at the boundary  $\partial B_{\e}$. Therefore 
the boundary $\partial B_{\e} \cap S_i$ of the irreducible curve $S_i$ is diffeomorphic to the link 
of the corresponding $\Sigma^{0}_{i}$, hence it  is diffeomorphic to a circle. 

We call this circle the {\it outside loop} of $S^{*}_i$, for any $i$, and denote by $U$ the set of outside loops of $S^{*}$.  Note that over any loop $u_{i} \in U$ we have a local system monodromy 
 $\nu_{u_{i}}:  \bZ^{\tmu_i} \rightarrow \bZ^{\tmu_i}$, and this monodromy is invariant in the admissible deformation.

Starting from the outside loop $u_{i}$, we retract the Riemann surface  with boundary $S^*_i$ to a bouquet configuration $\Gamma_i$ of  loops connected by simple paths which have a single base point $z_{i}\in S^*_i$.  These loops will be indexed by the set $W_{i}$. Let us count how many loops are in $W_{i}$.  There are $2g_{i}$ so-called {\it genus loops}, where $g_{i}$ is the genus of the normalisation of the compact singular Riemann surface obtained from $S_{i}$ by contracting its circle boundary to one point. Besides those loops, there are  $\tau_{i} = \sum_{q\in Q_{i}}\# K_q$ loops, where the set $K_q$ introduced at \S\ref{ss:index}  is indexing the local branches of $S_{i}$ at $q\in Q_{i}$. We thus have $\# W_{i} = 2g_{i} + \tau_{i}$.

The pair $(T_{S}^{*}, T_{S}^{*}\cap F^{(1)})$ decomposes into the connected pairs $(T_{i}^{*}, T_{i}^{*}\cap F^{(1)})$ along the curve components $S^*_i$.  Let us denote the projection of the tubular neighborhood by $\pi_S : T_{S}^{*} \to S^*$.
 Each  pair $(T_{i}^{*}, T_{i}^{*}\cap F^{(1)})$ is then homotopy equivalent
 to the pair  $(\pi_S^{-1}(\Gamma_i),  \pi_S^{-1}(\Gamma_i) \cap F^{(1)})$.  For each
 loop in the bouquet of loops $\Gamma_i$,  we have a pair similar to $(\cZ_k,\cC_k)$ defined above.
The difference is that the pairs $(\cZ_k,\cC_k)$ are disjoint whereas in $S^*_i$ 
the loops meet at a single point $z_i$. We therefore  take as reference the transversal
 fiber $\tF_i  := F^{(1)} \cap\pi_\Sigma^{-1}(z_i)$ over the point $z_i$. 

\medskip

For any $w\in W_i$, we  denote by $\nu_w$ the vertical monodromy along the loop in $S^{*}_{i}$ indexed by $w$.

Let 
\be \label{eq:hatA}
\hat{A}_{i} : \pi_{1}(S^{*}_{i}, z_{i}) \to \Aut (H^{n-2}(F^{\pitchfork}_{i};\bZ))
\ee
be the vertical monodromy representation, and let $H^{n-2}(F^{\pitchfork}_{i})^{\hat{A}_{i}}$
 denote the submodule of invariants. 
 Let us point out that, in general, this representation cannot be related by a Lefschetz slicing argument\footnote{i.e. the surjectivity of the map $\pi_{1}(S^{*}_{i}, z_{i}) \to \pi_1(\Sigma_{s,i}, z_{i})$ induced by inclusion.}
  to the monodromy representation defined at \eqref{vm}, because the Lefschetz argument needs the rectified homotopical depth condition for $\Sing f$ (see \cite{HL}, \cite[Sect. 9 and 10]{Ti-book}).  
However,  if the singular set has maximal rectified homological depth, then by \cite[Thm. 9.3.1]{Ti-book}, the equality 
$$H^{n-2}(F^{\pitchfork}_{i})^{A_{i}} = H^{n-2}(F^{\pitchfork}_{i})^{\hat{A}_{i}}$$
holds for any $i\in I$.
 In view of Theorem \ref{c14a}(c) and of the forthcoming Corollary \ref{c:no1diminside}, we conjecture that this later equality  holds in general.
\begin{lemma}\label{p:concentr}
\begin{enumerate}
\item  
We have the isomorphisms:
\[\begin{split}
H^{n-1}(T_{i}^{*}, T_{i}^{*}\cap F^{(1)}) \cong  \bigcap_{w\in W_i} \ker (\nu_w-\id \mid  H^{n-2}(\tF_i)),  \\
H^{n}(T_{i}^{*}, T_{i}^{*}\cap F^{(1)}) \cong \coker ( \bigoplus_{w\in W_i}(\nu_w-\id \mid  H^{n-2}(\tF_i)) ).
\end{split} \]
\item
 The Mayer-Vietoris exact sequence \eqref{eq:mv} is trivial except of the  6-terms sequence:
 \begin{equation}\label{eq:long} \small
 \begin{split}
 0\to H^{n-1}(T_S,T_S \cap F^{(1)})  \to \bigoplus_{i\in I} H^{n-2}(F^{\pitchfork}_{i})^{\hat{A}_{i}} 
 \oplus  \bigoplus_{q\in Q} H^{n-2}(F_q)  \stackrel{j}{\to}
  \bigoplus_{q\in Q} \bigoplus_{k\in K_q} {\ker} \; (\nu_k - \id) \\
 \to H^{n}(T_S,T_S \cap F^{(1)}) \to 
 \bigoplus_{i\in I}  \coker ( \bigoplus_{w\in W_i}(\nu_w-\id\mid  H^{n-2}(\tF_i)) ) \oplus \bigoplus_{q\in Q} H^{n-1}(F_{q}) \\
\to  \bigoplus_{q\in Q} \bigoplus_{k\in K_q} {\coker} \; (\nu_k - \id \mid  H^{n-2}(\tF_k)) \to  0.
\end{split}
\end{equation}

\item The following Euler characteristic formulae hold:
\begin{equation}\label{eq:euler0}
\chi(T_{i}^{*}, T_{i}^{*}\cap F^{(1)}) = (-1)^{n}(2g_i +\tau_i -1) \tmu_i,
\end{equation}
where $g_{i}$ is the genus of the normalization of $S_{i}$  and $\tau_i$ is the total  number of the local branches of $\Sigma_{i}$ at the points $Q_{i}$, \\
and:

\begin{equation}\label{eq:euler}
 \chi(T, T\cap F^{(1)}) =   - \sum_{q\in Q} (\chi(F_q)-1) +(-1)^{n} \sum_{i\in I} (2g_i +\tau_i -1) \tmu_i + (-1)^{n}\sum_{r\in R} \mu_{r}.
\end{equation}
\end{enumerate}
\end{lemma}

\begin{proof}
From  \eqref{eq:slicing3b}, \eqref{eq:slicing4b} and \eqref{eq:retr} we know that $H^{j}(T_S,T_S \cap F^{(1)}) = 0$ for $j\le n-2$ and $j>n$.
By \eqref{eq:sumdecomp}   and \eqref{eq:wang2}, the cohomology
$$H^{*}(T_{S}^{*}\cap B_Q, (T_{S}^{*}\cap F^{(1)}) \cap (B_Q\cap F^{(1)})) \cong  \bigoplus_{q\in Q}  \bigoplus_{k\in K_q} H^{*}(\cZ_k, \cC_k)$$ 
is concentrated in dimensions $n-1$ and $n$, and its Euler characteristic is zero since it is a direct sums of relative circle bundles. And by \S \ref{ss:term}, the cohomology $H^{*}(B_Q, B_Q\cap F^{(1)})$ is also concentrated in dimensions $n-1$ and $n$.

  The cohomology $H^{*}(T_{S}^{*}, T_{S}^{*}\cap F^{(1)}) = \bigoplus_{i\in I}H^{*}(T_{i}^{*}, T_{i}^{*}\cap F^{(1)})$ is also concentrated  in degrees $n-1$ and $n$ since the pair  $(T_{i}^{*}, T_{i}^{*}\cap F^{(1)})$ is the relative bundle $(\pi_S^{-1}(\Gamma_i),  \pi_S^{-1}(\Gamma_i) \cap F^{(1)})$ over a bouquet of loops indexed by $W_{i}$ with  fibre $(\tE_i, \tF_i)$. More precisely we have a relative Wang sequence for a bundle over a wedge of circles, the nontrivial part of which is the following:
\begin{equation}\label{eq:commut}
 H^{n-1}(T_{i}^{*}, T_{i}^{*}\cap F^{(1)})  \hookrightarrow  
H^{n-1}(\tE_i, \tF_i)  
\xrightarrow{\bigoplus_{w\in W_{i}}(\nu_w- \id)}
\bigoplus_{w\in W_{i}} H^{n-1}(\tE_i, \tF_i)  
  \twoheadrightarrow  H^{n}(T_{i}^{*}, T_{i}^{*}\cap F^{(1)}) 
\end{equation}


The relative cohomology  group  $H^{n-1}(T_{i}^{*}, T_{i}^{*}\cap F^{(1)})$  identifies to the kernel of the boundary map $\partial = \bigoplus_{w\in W_{i}}(\nu_w- \id)$, and this kernel is the intersection $\bigcap_{w\in W_{i}}\ker(\nu_w- \id \mid  H^{n-2}(\tF_i))$ which is by definition  the submodule of invariants $H^{n-2}(F^{\pitchfork}_{i})^{\hat{A}_{i}}$ defined by the representation \eqref{eq:hatA}. This proves (a). Part (b) follows by using all the above  facts proved in the Mayer-Vietoris sequence \eqref{eq:mv}.

 \noindent 
 (c)  Counting the ranks in the exact sequence \eqref{eq:commut} yields the above claimed formula \eqref{eq:euler0}.
The second formula \eqref{eq:euler} is obtained by taking the 
Euler characteristic in the Mayer-Vietoris long exact sequence \eqref{eq:long}, and using Lemma \ref{p:concentr}(c).
\end{proof}


Let us give a quick consequence of the above study for the particular case $Q=\emptyset$, yielding a slight extension of Theorem \ref{c14a}(c):
 
\begin{corollary}\label{c:no1diminside}
Let $n\ge 3$ and $\dim_{0}\Sing_{\cW} f= 2$. If $\overline{\Sigma_{2}} \cap \Sigma_{1} =\emptyset$ then:
\[
\begin{split}
 H^{n-2}(F) \cong   \bigoplus_{i\in I} H^{n-2}(F^{\pitchfork}_{i})^{\hat{A}_{i}}, \hspace{3cm} \\
H^{n-1}(F)   \hookrightarrow    \bigoplus_{i\in I}  \coker ( \bigoplus_{w\in W_i}(\nu_w-\id \mid  H^{n-2}(\tF_i)) ) \oplus \bigoplus_{r\in R} H^{n-1}(F_{r}).
\end{split}
\]
\end{corollary}
   
\begin{proof}
     By the above construction we have 
$$(T_{S}, T_{S}\cap F^{(1)}) = \bigsqcup_{i\in I}(T_{i}^{*}, T_{i}^{*}\cap F^{(1)})$$
and our result follows by Lemma \ref{p:concentr}(b), by \eqref{eq:decomp} and from the Mayer-Vietoris sequence \eqref{eq:long}
which reduces to two isomorphisms since $Q=\emptyset$.
\end{proof}


\subsection{End of the computation of $H^{n-2}(F)$.}\ 

We continue the computation without any restrictions on the set $Q$. 

The function germ  $f_{| l=\eta}$ at some point $q\in Q$ is a function germ with a $1$-dimensional singularity, and we have denoted by $F_{q}$ its Milnor fiber.  The pair 
 $(T_{S}^{*} \cap B_Q, T_{S}^{*} \cap B_Q\cap F^{(1)})$ is a disjoint union of pairs localized at the points of $Q$. 
Let us denote  $(\cZ_{q}, \cC_{q}) := \bigsqcup_{k\in K_{q}} (\cZ_{k},\cC_{k})$, where  the pair $(\cZ_{k},\cC_{k})$
is defined in  \eqref{eq:sumdecomp} and corresponds to the  branch $S_{q,k}$ of the germ of the singular locus $(S,q)$.
We claim that we have the natural monomorphism\footnote{Remark that the monomorphism \eqref{eq:incl} can be viewed as a particular case of \eqref{upb}.} (induced by the inclusion of pairs):
\begin{equation}\label{eq:incl}
  \widetilde H^{n-2}(F_{q}) \cong  H^{n-1} (B_{q}, F_{q}) \stackrel{\iota}{ \hookrightarrow}  H^{n-1}(\cZ_{q}, \cC_{q})  \cong \bigoplus_{k\in K_q} \ker (\nu_k - \id \mid  \widetilde H^{n-2}(\tF_k))
\end{equation}
where the last isomorphism follows directly from \eqref{eq:wang2}, and where $\nu_{k}$ is here the vertical monodromy acting on the branch  $S_{q,k}$ of the curve germ $(S, q)$.  

To prove our claim, we apply  the results and notations of \S\ref{Tib} for $f$  to the function germ $g_{q} := f_{| l=\eta}$ at $q\in Q$, see in particular  Definition \ref{d:vanish} and the notations following it.  Its singular locus $\Sing_{\cW\cap \{l=\eta\}} g_{q}$  coincides with  $S$,  viewed as set germs at $q$.
We shall use the index $q$ to designate similar objects for $g_{q}$, namely  $T_{q}$ will be the tubular neighbourhood of the curve germ  $(S,q)$, and  $L_{q}$ the slice monodromy acting on the slice $F'_{q}$ of the Milnor fibre $F_{q}$ of  $g_{q}$.

By  \eqref{eq:monomorphism0}, we have the monomorphism:
\begin{equation}\label{eq:monomorphism1}
\widetilde{H}^{n-2}(F_{q}) \hookrightarrow  \ker \big[L_{q} -\id \mid  H^{n-1}(T_{q}, T_{q}\cap F'_{q})\big]
\end{equation}

The pair $(T_{q}, T_{q}\cap F'_{q})$ is a disjoint union, over  ${k\in K_q}$,  of the  transversal Milnor data
 $(B_{k, j}, F_{k}^{\pitchfork})$ at some  point of the branch  cut out by a generic transversal slice $\cH$ of the space germ $(X \cap l^{-1}(\eta), q)$ passing close enough to the point $q$. The number of those points where the generic slice intersects  the branch  $S_{q,k}$ is the local multiplicity $\mult_{q} (S_{q,k})$, and let us index these points by the set $\Lambda_{q,k}$.
   With these notations at hand, we have the isomorphism:  
\begin{equation}\label{eq:section-n-s}
 H^{n-1}(T_{q}, T_{q}\cap F'_{q}) \cong \bigoplus_{k\in K_q}\bigoplus_{j\in \Lambda_{q,k}} H^{n-1}(B_{k, j}, F_{k}^{\pitchfork}),
\end{equation}
 where the Milnor data $(B_{k, j}, F_{k}^{\pitchfork})$ are the same for all $j\in \Lambda_{q,k}$, up to homeomorphisms. 

The $l$-monodromy acts on the direct sum of \eqref{eq:section-n-s} for each fixed $k$,  and this action corresponds  to the cyclic movement of a singular point of $S_{q,k}\cap \cH$. 
  As explained in \cite{Si-mon},  \cite{Ti-iomdin},   \cite[Proof of Prop. 3.1]{Ti-nonisol},  this yields a particular shape of the matrix of the monodromy matrix $L$, which gives the following direct sum splitting:
  $$\ker \big[L_{q} -\id \mid  H^{n-1}(T_{q}, T_{q}\cap F'_{q})\big]
  \cong    \bigoplus_{k\in K_{q}}  \ker (\nu_{k} -\id \mid  H^{n-2}( F_{k}^{\pitchfork})).$$
By comparing it with \eqref{eq:monomorphism1}, this finishes the proof of our claimed monomorphism \eqref{eq:incl}.
  
\medskip 

Coming back to the global picture, let us point out that in the direct sum of \eqref{eq:incl}, each fiber  $\tF_k$  identifies with the transversal  fiber $\tF_i$ for the  $i\in I$ corresponding to $k$ in the decomposition {$K_q = \bigsqcup_{i\in I} K_{qi}$.
Let us denote by $\iota_{q,k}$ the composition of the injection $\iota$ from \eqref{eq:incl} with the projection on the direct summand ${\ker} (\nu_k - \id \mid  H^{n-2}(\tF_k))$.

By the exact sequence \eqref{eq:mv} via Lemma \ref{p:concentr}(a) and the exact sequence \eqref{eq:long}, and from the expression of the second direct summand given in  \S \ref{ss:term}, we have $H^{n-2} (F)\cong \ker  j$, where $j$ is the following morphism which occurs in \eqref{eq:long}:

\begin{equation}\label{eq:hook}
\bigoplus_{i\in I} H^{n-2}(F^{\pitchfork}_{i})^{\hat{A}_{i}} \oplus  \bigoplus_{q\in Q} H^{n-2}(F_q) \stackrel{j}{\longrightarrow}   \bigoplus_{q\in Q} \bigoplus_{k\in K_q} {\ker} \; (\nu_k - \id \mid  H^{n-2}(\tF_k)).
  \end{equation}
Let $$j_{1}: \bigoplus_{i\in I} H^{n-2}(F^{\pitchfork}_{i})^{\hat{A}_{i}} \to  \bigoplus_{q\in Q} \bigoplus_{k\in K_q} {\ker} \; (\nu_k - \id \mid  H^{n-2}(\tF_k))$$ denote the restriction of $j$ to the first summand. We will denote by $j_{2}$ the restriction of $j$ to the second summand. 
By Lemma  \ref{p:concentr}(a) and (b), $j_{1}$ identifies with
the map:
\begin{equation}\label{eq:mapj1}
  \bigoplus_{i\in I}\bigcap_{w\in W_{i}}\ker (\nu_w- \id \mid  H^{n-2}(\tF_i)) \stackrel{j_{1}}{\longrightarrow} \bigoplus_{q\in Q} \bigoplus_{k\in K_q} {\ker} \; (\nu_k - \id \mid  H^{n-2}(\tF_k)),
\end{equation}
which is the diagonal map, in the sense that, for each fixed $i\in I$ such that $Q_{i}\not= \emptyset$, the intersection from the left side injects into each of the 
members of the direct sum from the right taken over $q\in Q_{i}$.
Indeed, if $Q_{i}\not= \emptyset$,  the intersection $\bigcap_{w\in W_{i}}$  of kernels from the left side of \eqref{eq:mapj1} injects in each kernel $\ker (\nu_k- \id)$ from the right side, for any $k\in K_{q}$ and  $q\in Q_{i}$. 
On the other hand, if $Q_{i}= \emptyset$, then $\bigcap_{w\in W_{i}}\ker (\nu_w- \id \mid  H^{n-2}(\tF_i))$ is clearly included in $\ker j$. 

So, if  $I_{0} := \{ i\in I\mid Q_{i}=\emptyset \}$, let $j'_{1}$ denote the restriction of $j_{1}$ from \eqref{eq:mapj1} where in the source we take only  the direct sum  $\bigoplus_{i\in I\m I_{0}}$ corresponding to the indices $I\m I_{0}$.

Let us denote by $j_{1}(H^{n-2}(F^{\pitchfork}_{i})^{\hat{A}_{i}})$  the subgroup in $\ker (\nu_k- \id \mid  H^{n-2}(\tF_i))$ for every $k\in K_{qi}$ and $q\in Q_{i}$. With all these notations and preliminaries we obtain the following:

\smallskip
\begin{theorem}\label{t:euler}
 Let $(X,0)$ be a pure $n+1$ dimensional space germ with $\r(X,\bZ)=n+1$,  $n\ge 3$ and $\dim \Sing_{\cW} f = 2$. Then:
 \begin{enumerate}
\item $H^{n-2}(F)$ is free. The morphisms $j'_{1}$ and $j_{2}$ are injective. 
\item   
$H^{n-2}(F) \cong  G \oplus \bigoplus_{i\in I_{0}} H^{n-2}(F^{\pitchfork}_{i})^{\hat{A}_{i}},$
 where  $G\cong \im j'_{1} \cap \im j_{2}$ is a submodule of
  $$\bigoplus_{i\in I\m I_{0}} \big[ j_{1}(H^{n-2}(F^{\pitchfork}_{i})^{\hat{A}_{i}})\cap \bigcap_{q\in Q_{i},  k\in K_{qi}}\iota_{qk}(H^{n-2}(F_q)) \big].
 $$
\end{enumerate}
\end{theorem}
 
\begin{proof}
\noindent
(a). We have $H^{n-2} (F)\cong \ker  j$, where $j$ is given in \eqref{eq:hook}, and this is  consequence of the assumption $\r(X,\bZ) = n+1$. The freeness follows then from the fact that the source of $j$ is a free $\bZ$-module, since $H^{n-2}(F^{\pitchfork}_{i})^{\hat{A}_{i}}$ are free, and $H^{n-2}(F_q)$ are also free due to \eqref{eq:incl}, both being consequences of the assumption $\r(X, \bZ) = n+1$ (see also the proof of Theorem \ref{c14a}(d)).

 We have seen before that $j'_{1}$ is injective and it is actually the diagonal map for each fixed $i\in I$. 
Let us now consider 
$$j_{2}: \bigoplus_{q\in Q} H^{n-2}(F_q) {\to}   \bigoplus_{q\in Q} \bigoplus_{k\in K_q} {\ker} \; (\nu_k - \id \mid  H^{n-2}(\tF_k))$$
 as the restriction of $j$ to the second summand.
After taking the direct sum $\bigoplus_{q\in Q}$ at both sides of \eqref{eq:incl}, we deduce the injectivity of $j_{2}$.\\

\noindent
(b). By the same argument as for Corollary  \ref{c:no1diminside}, we get the inclusion $\ker j \supset \bigoplus_{i\in I_{0}} H^{n-2}(F^{\pitchfork}_{i})^{\hat{A}_{i}}$.  Since $j'_{1}$ and $j_{2}$ are injective, we deduce:  
$$\ker j \cong  (\im j'_{1} \cap \im j_{2})   \oplus \bigoplus_{i\in I_{0}} H^{n-2}(F^{\pitchfork}_{i})^{\hat{A}_{i}}.$$

 Next, for 
any fixed $i\not\in I_{0}$, let us consider 
the following restriction, where in the target we take the summands over  $K_{qi}$ only:
\[ j_{|} : H^{n-2}(F^{\pitchfork}_{i})^{\hat{A}_{i}} \oplus  \bigoplus_{q\in Q_{i}} H^{n-2}(F_q) 
{\longrightarrow}   \bigoplus_{q\in Q_{i}} \bigoplus_{k\in K_{qi}} {\ker} \; (\nu_k - \id \mid  H^{n-2}(\tF_k))
\]
Its kernel is then isomorphic $j_{1}(H^{n-2}(F^{\pitchfork}_{i})^{\hat{A}_{i}})\cap \bigcap_{q\in Q_{i} k\in K_{qi}}\iota_{qk}(H^{n-2}(F_q))$, due to the fact that $j_{1}$ is the diagonal map.

Nevertheless, when we take the direct sum over $i\in I\m I_{0}$,  the kernels may not add up since there is interaction between different 
$S_{i}$ due to the fact that whenever $q\in Q$ belongs to several irreducible components $S_{i}$,
then the fiber $F_{q}$ contributes with $\iota_{qk}(H^{n-2}(F_q))$ for all those indices $k$ such that $k\in K_{qi}$. Therefore we do not get $\ker j$ as a direct sum but as a submodule only; our claim is proved.
\end{proof}
 
We immediately get from the above theorem the following:
\begin{corollary} \label{p:conc} Let $n\ge 3$. Let $I_{1} := \{ i\in I \m I_{0}\mid H^{n-2}(F_q)=0 \mbox{ for some } q\in Q_{i}\}$. 
\begin{enumerate}
\item
$G$ is isomorphic to a submodule of  $\bigoplus_{\substack{q\in Q_{i} \\ i\not\in I_{1}}} H^{n-2}(F_q)$. 
In particular, if $I_{1}=I$,  then $H^{n-2}(F) =0$.
\item 
$b_{n-2}(F) \le \sum_{i\in I \m I_{1}} \min \{ b_{n-2}(F^{\pitchfork}_{i}),    b_{n-2}(F_{q}) \mid q\in Q_{i}\}.$
\end{enumerate} \fin
\end{corollary}


\section{Examples}\label{examples}
In this section, we apply our results to specific computations on examples. Let us refer here again to \cite{Za, Ne2}, not only for very interesting and sharp results about the Milnor fibre in a  class of functions with  $s=2$, but also for examples of computations in homology.  More computations in homology are done in \cite{ST-deform} in case of admissible deformations of $1$-dimensional singularities, which case, as we have already mentioned before, has common grounds with our case $s=2$, whereas the computations in cohomology are different.

\begin{example}\label{rem010}
If the singularity at the origin of $f:(X,0) \to (\bC,0)$ is contained in a positive dimensional stratum $S$, say of complex dimension $r \leq s<n$, then by the local product structure around the origin we get:
$$\widetilde{H}^{n-r+1}(F) \cong \cdots \cong \widetilde{H}^{n}(F) \cong 0.$$
So in this case the reduced cohomology of the Milnor fiber $F$ is concentrated in degrees $[n-s, n-r]$.
In terms of vanishing cycles, this can be seen from the {\it base change property} (e.g., see \cite[Lemma 4.3.4]{Sc}) as follows: if $H$ is a normal slice (in the ambient space $\bC^N$) through the origin to the stratum $S$ containing the origin, then we have for $f':=f\vert_H$ and with $\cGb$ a bounded constructible complex that:
\be
\varphi_{f'}(\cGb\vert_H) \cong (\varphi_f\cGb)\vert_H.
\ee 

Let us consider the example $f:\bC^4\to \bC$ given by $f=xyz$.
The singular locus of $f$ is given by 
$\Sigma=H_{xy} \cup H_{xz} \cup H_{yz},$
with $H_{xy}:=\{x=y=0\}$,  $H_{xz}:=\{x=z=0\}$, $H_{yz}:=\{y=z=0\}$. These three complex 2-planes intersect mutually along  the line $H_{xyz}:=\{x=y=z=0\}\cong \bC$. Then $\Sigma$ has a Whitney stratification with three $2$-dimensional strata
$$\Sigma_{2,1}:=H_{xy} \setminus H_{xyz}, \ \Sigma_{2,2}:=H_{xz} \setminus H_{xyz}, \ \Sigma_{2,3}:=H_{yz} \setminus H_{xyz},$$
each homotopy equivalent to a circle, 
and a $1$-dimensional stratum $\Sigma_1=H_{xyz}$. 
The transversal type of each of the three $2$-dimensional strata is A$_{1}$, and each vertical monodromy $A_i$ is the identity.  Formula \eqref{upb} yields in this case that $H^1(F) \hookrightarrow \bZ^3$. It appears that this inclusion is strict, which also shows that the monomorphism in \eqref{upb} is not in general an isomorphism. Indeed, since the singularity of $f$ at the origin is contained in a $1$-dimensional stratum, by slicing with a generic $\bC^3$ through the origin reduces the calculation of $H^1(F)$ to the case of the $1$-dimensional singularity defined by $xyz=0$ in $\bC^3$, for which the origin is a $0$-dimensional stratum. This was considered by Siersma in \cite{Si3}, who computed that  the Milnor fiber $F$ at the origin is homotopy equivalent to $S^1\times S^1$, hence $H^1(F) \cong \bZ^2$. 
This calculation can also be performed via Theorem \ref{t:euler}(a): indeed, we have  $j_{1}= \id$, and hence $H^1(F) \cong \ker j \cong \im ( j_{2}: \bZ^{2}\to \bZ^{3}) \cong \bZ^{2}$, since $j_{2}$ is injective.
\end{example}

\begin{example}
Let $f:\bC^4\to \bC$ be given by $f=xyzu$.
The singular locus of $f$ is $2$-dimensional and is given by 
$$\Sigma=H_{xy} \cup H_{xz} \cup H_{xu} \cup H_{yz}  \cup H_{yu}  \cup H_{zu}, $$
where $H_{xy}:=\{x=y=0\}$, etc, thus each component is a $2$-plane in $\bC^4$. Any two of the six components of $\Sigma$ intersect either along a complex line, e.g., $H_{xy} \cap H_{xz}=H_{xyz}=\{x=y=z=0\}$ is the $u$-line, or at the origin, e.g., $H_{xy} \cap H_{zu}=\{(0,0,0,0)\}$. Moreover, any three (or more) of the six components of $\Sigma$ intersect at the origin, e.g.,  $H_{xy} \cap H_{xz} \cap H_{xu}=H_{xyzu}=\{x=y=z=u=0\}=\{(0,0,0,0)\}$. 
A Whitney stratification of $\Sigma$ can be given with six $2$-dimensional strata, one for each component of $\Sigma$, which are of the form
$$\Sigma_{2,1}:=H_{xy} \setminus ( H_{xyz} \cup H_{xyu} ) \cong \bC^* \times \bC^*,$$
etc. There are four $1$-dimensional strata of the form $H_{xyz} \setminus H_{xyzu} \cong \bC^*$, and the origin $H_{xyzu}=\{(0,0,0,0)\}$ is a $0$-dimensional stratum.

Each of these $2$-dimensional strata has the homotopy type of $S^1 \times S^1$ and $A_1$-transversal type. 
The two generators of the fundamental group $\bZ^2$ of a $2$-dimensional stratum act trivially on the first cohomology group, $\bZ$, of the corresponding transversal Milnor fiber, so the vertical monodromy along each of the $2$-dimensional strata is the identity.  
Formula \eqref{upb}  yields the inclusion $H^1(F) \hookrightarrow \bZ^6$.

Let us further use Theorem \ref{t:euler}.
A generic slice $l=\eta$ has 1-dimensional singularities. We have that $S$ is a configuration of  6 lines intersecting at 4 points, where 3 lines pass through each of the 4 points.\footnote{One may refer to \cite[\S 5.4]{ST-deform} for a related computation in homology in case of an admissible deformation of a 4 planes central arrangement.} Thus $S_{i}^{*} \simeq \bC^{**}$, $i=1,\ldots, 6$,  and each local vertical monodromy around such a puncture is the identity.  We compute $j$ and find that the image of $j_{1}$ has 6 generators. Intersecting it with 
the image of $j_{2}$ introduces 3 relations among these generators, more precisely a symmetric linear relation between 3 generators in each of the 4 punctures, and resolving the system  yields finally 3 linear relations among the 6 generators. We get $H^1(F) \cong \ker j \cong \bZ^{3}$. 

On the other hand,  by direct computation, the fibre $F=\{xyzu=1\}$ is homotopy equivalent to 
$S^1\times S^1 \times S^1$, hence $H^1(F) \cong \bZ^3$, confirming the above result. 
\end{example}

\begin{example}
Let $f:\bC^4\to \bC$ be given by $f=x^2z+y^2u$.
The hypersurface $f=0$ has a $2$-dimensional singular locus $\Sigma=\{x=y=0\}$, with Whitney strata $\Sigma_0=\{(0,0,0,0)\}$, $\Sigma_1=\{(0,0,0,u) \mid u\neq 0\} \cup \{(0,0,z,0) \mid z\neq 0\}$ and $\Sigma_2=\{(0,0,z,u) \mid z \neq 0, u \neq 0\}$. The transversal Milnor fiber $F^\pitchfork$ to the stratum $\Sigma_2$ is the Milnor fiber of the singularity at $(0,0)$ of the curve $x^2+y^2=0$ in $\bC^2$, so $F^\pitchfork\simeq S^1$. It follows from Theorem \ref{c14a}(a) that if $F$ denotes the Milnor fiber of $f$ at the origin, then
$$H^1(F) \hookrightarrow \bZ.$$
On the other hand, it can be seen by using the Thom-Sebastiani theorem that $F\cong S^1 \ast S^1 \simeq S^3$, so in fact $H^1(F) =0$.

In the setting of Theorem  \ref{t:euler}, slicing with the hyperplane $u+z=1$ and taking advantage of the homogeneity which implies that all the fibrations are global, we get $S^{*}\simeq \bC^{**}$, where the vertical monodromies $\nu_{i}$ around the two special points are $-\id$. This implies that $\ker (\nu_{i}-\id) = 0$ and therefore Theorem \ref{t:euler}(b) yields indeed that $H^1(F) =0$.
\end{example}

\begin{example}\label{ex77}
Let $f:\bC^4\to \bC$ be given by $f=x^2+x(y^2+z^2+u^2)$.
The singular locus of the hypersurface $f=0$ has a Whitney stratification with a zero-dimensional stratum $\Sigma_0=\{(0,0,0,0)\}$ and a two-dimensional stratum 
$$\Sigma_2=V(x,y^2+z^2+u^2) \setminus \{(0,0,0,0)\}.$$ 
The transversal Milnor fiber $F^\pitchfork$ to the stratum $\Sigma_2$ is the fibre of a $A_{1}$ singularity. The stratum $\Sigma_2$ is homotopy equivalent to the link of a quotient surface singularity of type $A_1$, hence $\pi_1(\Sigma_2)=\bZ/2$. 
Since there are no $1$-dimensional strata, Theorem \ref{c14a}(c) shows that the first cohomology group of the Milnor fiber $F$ of $f$ at the origin is computed as:
$$H^1(F) \cong  H^1(F^\pitchfork)^{\bZ/2}.$$
Note that we can write $f=P(h,g)$, with $h=x:\bC^4 \to \bC$, $g=y^2+z^2+u^2:\bC^4 \to \bC$ and $P=h^2+hg$, so the Milnor fiber $F$ at $0$ can be deduced from \cite[Theorem A]{Ne} as having the homotopy type of $S^1\vee S^3 \vee S^3$, which gives $H^1(F)\cong \bZ$. Alternatively, after a change of coordinates, we note that $F$ is the Milnor fiber of the singularity at the origin of the polynomial  $x^2+(y^2+z^2+u^2)^2$, and its homotopy type can be easily deduced via the Thom-Sebastiani theorem as the suspension on two disjoint $S^2$'s.
 
Let us indicate how this isomorphism follows from Theorem \ref{t:euler}(b). By this result  we get $H^1(F) \cong  H^1(F^\pitchfork)^{\bZ}$ since the stratum $S$ is the complex link of  
 $\Sigma_2$ hence homotopy equivalent to a circle. We  compute the vertical monodromy $\nu$ of $H^1(F^\pitchfork)$ along $S$. We first slice near the origin with $\{ l:= u =\eta\}$ and consider the restriction $f\vert_{u=\eta}= f'(x,y,z) = x^{2} + x(y^{2}+z^{2} + \eta^{2})$ with singular locus $S= Z(x,y^2+z^2+\eta^2)$. We consider a circle $y^2+z^2 = t>0$ (in real coordinates) which is a geometric generator of the fundamental group of $S$. The monodromy action $\nu$ on  $H^1(F^\pitchfork) \cong \bZ$ along this circle is the identity, thus  $H^1(F) \cong \ker (\nu -\id \mid H^1(F^\pitchfork)) = \bZ$. 
\end{example}

More generally, consider the following example.
\begin{example}\label{ex77b}

Let $f:\bC^4\to \bC$ be given by $f=x^p+(y^2+z^2+u^2)^q$, with $p,q \geq 2$. This is very similar to Example \ref{ex77} above, in which $p=q=2$. Indeed, the singular locus of the hypersurface $f=0$ has a Whitney stratification with a zero-dimensional stratum $\Sigma_0=\{(0,0,0,0)\}$ and a two-dimensional stratum $$\Sigma_2=V(x,y^2+z^2+u^2) \setminus \{(0,0,0,0)\}.$$ 
Outside of the origin in $\bC^4$ the functions $x$ and $g=y^2+z^2+u^2$ are part of a coordinate system and define the smooth singular set.
\noindent 
The transversal Milnor fibre may be found as follows: at the two points of the 
set $\Sigma_{2} \cap \{y=z= a\} = \{ x=0, y=z= a, u^{2}+2a^{2} =0\}$, for some $a\not= 0$, we consider two Milnor balls in the ``vertical'' slice  $\{y=z= a\}$.
At each of the two points, one has local coordinates $v:=x$ and $w:=g$ in the slice $\{y=z= a\}$.  The transversal Milnor fibres at the two points are both described by the equation  $x^p+(2a^{2}+u^2)^q =\eta$. 
Therefore, the transversal singularity of $f$ at a point in $\Sigma_2$ is locally expressed by 
$x^p+g^q=0$. So the transversal Milnor fiber $F^\pitchfork$ to the stratum $\Sigma_2$ is a bouquet of $(p-1)(q-1)$ circles $S^1$, with $H^1(F^\pitchfork) \cong  \bZ^{(p-1)(q-1)}$.  

\noindent 
Once again,
$\pi_1(\Sigma_2)=\bZ/2$, and since there are no $1$-dimensional strata, Theorem \ref{c14a}(c) shows that the first cohomology group of the Milnor fiber $F$ of $f$ at the origin is computed as:
$$H^1(F) \cong  H^1(F^\pitchfork)^{\bZ/2} \hookrightarrow \bZ^{(p-1)(q-1)}.$$
Let us also note that an application of the Thom-Sebastiani theorem shows that 
the Milnor fiber $F$ of $f$ at the origin is homotopy equivalent to the join of $p$ points (corresponding to $x^p=1$) with a disjoint union of $q$ spheres $S^2$ (each given by $y^2+z^2+u^2=1$). In particular, $H^1(F) \cong \bZ^{(p-1)(q-1)}$.
This implies that $\pi_1(\Sigma_2)$ acts trivially on $H^1(F^\pitchfork)$. 

\noindent 
 In order to apply Theorem \ref{t:euler}(b), we compute the vertical monodromy on the stratum 
 $S$ which is by definition the complex link  of the 2-dimensional stratum $\Sigma_{2}$. 
It may be defined by  $y^2+z^2 = \e$,  thus it is homotopy equivalent to a circle. A geometric generator of its
$\pi_{1}$ may be given by the same equation in real coordinates, considering $\e>0$ as a real constant.

While the real coordinates $y=\e \cos t , z = \e \sin t$, for $t\in [0, 2\pi[$, describe this circle, each of the two transversal Milnor balls  move along with fixed coordinates $x$ and $u$. This yields in particular a geometric monodromy which is the identity on all coordinates. 

\end{example}

 \medskip

\noindent{\bf Acknowledgements.} We thank Dirk Siersma and J\"org Sch\"urmann for their interesting and helpful remarks. Morihiko Saito sent us comments on an earlier version, we thank him for the constructive side.



\end{document}